\documentclass[12pt]{amsart}

\usepackage[all]{xy}
\usepackage{amscd,amsmath,amsthm,amssymb}
\usepackage{mathtools}
\usepackage{mathrsfs}
\usepackage{tikz}

\definecolor{cadmiumgreen}{rgb}{0.0, 0.42, 0.24}
\usepackage[
colorlinks, citecolor=cadmiumgreen,
pagebackref,
pdfauthor={Robin de Jong}, 
pdfstartview ={FitV},
]{hyperref}
\hypersetup{
pdftitle={Frobenius' theta function and Arakelov invariants in genus three},
}

\usepackage[
backrefs,
msc-links,
nobysame,
lite,
]{amsrefs}

\usepackage[left=3.5cm,top=3.5cm,right=3.5cm]{geometry}
\newtheorem{thm}{Theorem}[section]
\newtheorem{prop}[thm]{Proposition}

\newtheorem{lem}[thm]{Lemma}

\theoremstyle{remark}
\newtheorem{remark}[thm]{Remark}

\theoremstyle{definition}
\newtheorem{definition}[thm]{Definition}

\makeatletter
\renewcommand*\env@matrix[1][*\c@MaxMatrixCols c]{
  \hskip -\arraycolsep
  \let\@ifnextchar\new@ifnextchar
  \array{#1}}
\newcommand{\vertiii}[1]{{\left\vert\kern-0.25ex\left\vert\kern-0.25ex\left\vert #1 
    \right\vert\kern-0.25ex\right\vert\kern-0.25ex\right\vert}}
\newcommand*\isom{\xrightarrow{\sim}}
\newcommand{\pair}[1]{\langle#1\rangle}
\newcommand{\divisor}{\operatorname{div}}

\newcommand{\Pic}{\operatorname{Pic}}

\newcommand{\Sp}{\operatorname{Sp}}

\def\opn#1#2{\def#1{\operatorname{#2}}}
\opn\Vor{Vor}

\def\uu{\mathcal{U}}
\def\rr{\mathbb{R}}

\def\zz{\mathbb{Z}}
\def\cc{\mathbb{C}}

\def\HH{\mathcal{H}}
\def\nn{\mathcal{N}}
\def\mm{\mathcal{M}}
\def\ll{\mathcal{L}}
\def\bb{\mathcal{B}}
\def\pp{\mathcal{P}}
\def\xx{\mathcal{X}}
\def\CC{\mathcal{C}}
\def\jj{\mathcal{J}}

\def\aa{\mathcal{A}}
\def\oo{\mathcal{O}}

\def\ee{\mathcal{E}}
\def\ff{\mathfrak{F}}

\def\hh{\mathcal{H}}
\def\Im{\mathrm{Im}\,}
\def\d{\mathrm{d}}

\def\id{\mathrm{id}}

\def\vert{\mathrm{vert}}

\def\a{a}

\def\ct{\mathrm{ct}}
\def\deldelbar{\partial \overline{\partial}}
\def\Ar{\mathrm{Ar}}

\def\Hdg{\mathrm{Hdg}}

\numberwithin{equation}{section}

\newcounter{nootje}
\subjclass[2010]{
\href{https://mathscinet.ams.org/msc/msc2010.html?t=11G50}{11G50},
\href{https://mathscinet.ams.org/msc/msc2010.html?t=14G40}{14G40},
\href{https://mathscinet.ams.org/msc/msc2010.html?t=14H15}{14H15},
\href{https://mathscinet.ams.org/msc/msc2010.html?t=14H40}{14H40},
\href{https://mathscinet.ams.org/msc/msc2010.html?t=14H42}{14H42},
\href{https://mathscinet.ams.org/msc/msc2010.html?t=14H45}{14H45}.}

\date{\today}

\begin{document}

\title[Frobenius' theta function and Arakelov invariants]{Frobenius' theta function and Arakelov invariants in genus three}

\author{Robin de Jong}
\address{Leiden University \\ PO Box 9512  \\ 2300 RA Leiden  \\ The Netherlands }
\email{\href{mailto:rdejong@math.leidenuniv.nl}{rdejong@math.leidenuniv.nl}}

\begin{abstract} We give explicit formulas for the Kawazumi-Zhang invariant and Faltings delta-invariant of a compact and connected Riemann surface of genus three. The formulas are in terms of two integrals over the associated jacobian, one integral involving the standard Riemann theta function, and another involving a theta function particular to genus three that was discovered by Frobenius. We review part of Frobenius' work on his theta function and connect our results with a formula due to Bloch, Hain and Bost describing the archimedean height pairing of Ceresa cycles in genus three.
\end{abstract}

\maketitle

\setcounter{tocdepth}{1}

\thispagestyle{empty}

\section{Introduction}
\renewcommand*{\thethm}{\Alph{thm}}

\subsection{Motivation and background} 

Let $X$ be a compact and connected Riemann surface of positive genus~$g$. Associated to $X$ one has an invariant $\varphi(X)$ as defined by Kawazumi  \cites{kawhandbook, kaw} and  Zhang \cite{zhgs}, and an invariant $\delta(X)$ as defined by Faltings \cite{fa}. The invariants $\varphi(X)$ and $\delta(X)$ can be expressed in terms of the spectral theory of the Laplacian for the Arakelov metric on $X$, and play a key role in the arithmetic intersection theory of curves over number fields \cites{fa, zhgs}. 

More specifically, the Kawazumi-Zhang invariant occurs as an archimedean contribution in a formula that relates the self-intersection of the relative dualizing sheaf of a stable arithmetic surface with the height of Gross-Schoen cycles \cite[Theorem~1.3.1]{zhgs}, and the Faltings delta-invariant occurs as an archimedean contribution in a (Noether type) formula that relates the self-intersection of the relative dualizing sheaf with the stable Faltings height \cite[Theorem~6]{fa}. Both invariants can be seen as natural real-valued functions on the moduli space $\mm_g$ of compact Riemann surfaces of genus~$g$.

As such they have attracted attention in string theory. For example, recently it was shown that the integral of the Kawazumi-Zhang invariant $\varphi$ in genus two  against the volume form of the Siegel metric over the moduli space $\mm_2$ appears in the low energy expansion of the two-loop four-graviton amplitude \cite{dhg}. This connection has inspired a detailed study of its asymptotic behavior in the low energy (or tropical) limit \cites{dhg, dhgp, dhgrpr}. It was shown in \cite{dhgrpr} that the invariant $\varphi$ in genus two is an eigenfunction with eigenvalue $5$ of the Laplace-Beltrami operator on $\mm_2$. With this, the value of the two-loop $D^6\mathcal{R}^4$ interaction could be calculated rigorously, and was verified to be in agreement with the value of the interaction that was predicted by S-duality and supersymmetry.

Given their canonical nature, and their appearance in formulas for various intrinsic heights related to curves over number fields, it is natural to ask for formulas for $\varphi(X)$ and for $\delta(X)$ that are as concrete as possible. A natural viewpoint is to try to obtain explicit formulas in purely classical terms related to a period matrix of $X$, in particular, theta functions, and integrals over the jacobian of such. 

In genus one, such an explicit formula for $\delta(X)$ was exhibited already by Faltings \cite[Section~7]{fa}, featuring the classical discriminant modular form $\Delta(\tau)$ of weight~$12$ and level one (the invariant $\varphi(X)$ vanishes in this case). Explicit formulas of the type described were given in genus two by Bost \cite{bost_cras} for the delta-invariant, and by the author \cite{dj_adm} for the Kawazumi-Zhang invariant. In \cite{pioline} both invariants in genus two are written in terms of certain theta lifts of weak Jacobi forms in genus one. Wilms \cite{wilms}    has extended the formulae from \cites{bost_cras, fa, dj_adm} to hyperelliptic Riemann surfaces $X$ of arbitrary genus (cf.\ \cite[Theorem~4.8]{wilms}). 

\subsection{Aim of this paper} The aim of the present paper is to extend the above results to the case where the genus of $X$ is equal to three. Our main result is Theorem~\ref{main}. The formulas for $\varphi(X)$ and $\delta(X)$ given there feature integrals, over the jacobian, of (a) the classical Riemann theta function $\theta$, and (b) a beautiful second order theta function $\phi$ that was discovered and studied by Frobenius \cite{frob}. It would be interesting to see whether our explicit formulas may be brought to bear on the study of three-loop superstring amplitudes, or can be used to calculate the height of Gross-Schoen cycles explicitly in genus three.

Part of our paper consists in a review of some of Frobenius' work on his theta function $\phi$. Among other things we will see that, in modern terms, Frobenius' theta function determines a {\em Siegel-Jacobi form} of {\em weight eight}, {\em index one} and {\em full level} over the universal complex abelian threefold (cf.\ Section~\ref{mod_prop}). Further, Frobenius' theta function $\phi$ can be written down as an explicit linear combination of squares of  theta functions with characteristics (cf.\ Theorem~\ref{thm_explicit_frob}). Finally, the restriction of $\phi$ to the universal jacobian in genus three defines the {\em universal difference surface} over $\mm_3$ (cf.\ Theorem~\ref{level_one}).

We refer to \cites{cfg, gsm, vdgvg} for other modern accounts of Frobenius' theta function. In each of {\em loc.\ cit.} the discussion is based on a connection with the well-known {\em Schottky-Igusa modular form} in genus four, which Frobenius discusses at the very end of his article. Our approach here will be somewhat different and based on the earlier sections of \cite{frob}. 
 
\subsection{Main result}

Let $X$ be a compact Riemann surface of genus three, and let $J$ be the jacobian of $X$.
We refer to Sections~\ref{theta_with_char} and~\ref{def_frob_theta} for definitions of the theta functions $\theta$ and $\phi$ on (the universal covering of) $J$. The functions $\theta$ and $\phi$ give rise to canonical {\em normalized theta functions} $\|\theta\| \colon J \to \rr $ and $\|\phi\| \colon J \to \rr$, cf.\  Section~\ref{normalization}. Let $\mu$ be the Haar measure on $J$, normalized to give $J$ unit volume. We then have well-defined integrals
\begin{equation} \label{def_jacobian_invariants} \log \|H\|(X)= \int_{J} \log \|\theta\| \, \mu \, , \quad \log \|K\| (X)=  \int_{J} \log \|\phi\|\, \mu \, , 
\end{equation}
with values in $\rr$. Both $\log \|H\|$ and $\log\|K\|$ can be viewed as real-valued functions on the moduli space $\mm_3$.
\begin{thm} \label{main} 
Let $X$ be a compact and connected Riemann surface of genus three. Let $\varphi(X)$ be the Kawazumi-Zhang invariant of $X$, and let $\delta(X)$ be the Faltings delta-invariant of $X$. One has
\begin{equation} \label{explicit_phi}
\varphi (X)= -\frac{2}{3}\log \|K\|(X) + \frac{32}{3} \log \|H\| (X) + 8 \log 2 \, ,
\end{equation}
\begin{equation} \label{explicit_delta}
\delta(X) = -\frac{4}{3}\log \|K\| (X) - \frac{8}{3} \log\|H\|(X) - 24 \log(2\pi) + 16 \log 2 \, . 
\end{equation}
\end{thm}
We note that the invariant $\log\|H\|(X)$ is defined in any genus $g \geq 1$, whereas the invariant $\log\|K\|(X)$ is special to genus three. The invariant $\log \|H\|(X)$ was first introduced and studied by Bost \cite{bost_cras} and Bost, Mestre and Moret-Bailly \cite{bmmb} in genus two. Wilms \cite[Theorem~1.1]{wilms} has shown in any genus $g \geq 1$ the fundamental formula
\begin{equation} \label{wilms}
\delta(X) = -24 \log \|H\| (X)  + 2\, \varphi(X) -8g \log(2\pi) \, . 
\end{equation}
Wilms' result (\ref{wilms}) immediately shows that the two identities (\ref{explicit_phi}) and (\ref{explicit_delta}) in Theorem~\ref{main} are equivalent.

\subsection{Hain-Reed beta-invariant} 

Let $g \geq 3$ be an integer. The {\em Hain-Reed invariant} $\beta$ is a designated element of the space $C^0(\mm_g,\rr)/\rr$ of continuous real-valued functions modulo constants on $\mm_g$, defined and studied by Hain and Reed in \cite{hrar}. We refer to Section~\ref{hainreed_bis} for a brief discussion. In {\em loc.\ cit.}, an explicit formula for $\beta$ was announced in the case $g=3$ in terms of the Siegel modular form $\chi_{18}$ and a certain integral. However, as far as we can see, this formula has not been published since. 

As was shown by the author  \cite{djsecond}, the invariant $\beta$ admits a simple expression in terms of $\varphi$ and $\delta$. We conclude using Theorem~\ref{main} that we may write the invariant $\beta$ in genus three explicitly in terms of the two integrals $\log \|K\|$ and $\log \|H\|$.  More precisely, we have the following. Let $X$ be a compact and connected Riemann surface of genus $g$. We then set, following \cite[Section~1.4]{zhgs} and \cite[Definition~1.2]{djsecond}, 
\begin{equation} \label{def_lambda} 
 \lambda(X) = \frac{g-1}{6(2g+1)} \varphi(X) + \frac{1}{12}\delta(X) - \frac{g}{3} \log(2\pi) \, .
\end{equation}
Here $\varphi(X)$ and $\delta(X)$ are the Kawazumi-Zhang and Faltings delta-invariant of $X$ as before. By \cite[Theorem~1.4]{djsecond} the function $(8g+4)\lambda$ is a representative of the Hain-Reed invariant $\beta$ on $\mm_g$. Specializing to the case $g=3$ one finds that $\frac{4}{3}\varphi(X) + \frac{7}{3}\delta(X)$
is a representative of the Hain-Reed invariant $\beta$ on $\mm_3$. From the identities (\ref{explicit_phi}) and (\ref{explicit_delta}) we may then deduce the following.
\begin{thm} \label{explicit_beta} 
Let $\mm_3$ denote the moduli space of compact Riemann surfaces of genus three.
The function $-4\log\|K\|+8\log\|H\|$ is a representative in $C^0(\mm_3,\rr)$ of the Hain-Reed invariant $\beta$ on $\mm_3$.
\end{thm}

\subsection{Structure of the paper} 
In Section~\ref{Riemann} we review some basic definitions concerning characteristics and theta functions.
In Section~\ref{modularity} we discuss moduli of abelian varieties and in particular review the notions of Siegel modular forms and Siegel-Jacobi forms.
In Section~\ref{Frobenius} we introduce Frobenius' theta function and discuss some of its properties, following Frobenius' paper \cite{frob}.
In Section~\ref{sec:ar_geom_A_g} we introduce several canonical hermitian metrics and differential forms on moduli of abelian varieties, as well as canonically   normalized real-analytic versions of Siegel modular and Siegel-Jacobi forms.
In Section~\ref{arakelov} we review some identities between canonical differential forms on various moduli spaces of curves.  
In Section~\ref{sec:normalized} we specialize to genus three and prove Theorem~\ref{main}.
Finally in Section~\ref{bloch} we connect our findings with a formula due to Bloch, Hain and Bost describing the archimedean height pairing of Ceresa cycles in genus three.

\subsection{Acknowledgments} I would like to thank Boris Pioline for helpful correspondence and Emre Sert\"oz for showing me a nice proof of Lemma~\ref{difference}.

\renewcommand*{\thethm}{\arabic{section}.\arabic{thm}}

\section{Theta functions} \label{Riemann}

The purpose of this preliminary section is to review some fundamentals concerning characteristics and theta functions. 
A basic reference for this section is \cite[Chapter~II]{tata1}.
 We view all vectors as column vectors. Let $g \geq 1$ be an integer.

\subsection{Characteristics} \label{chars}   An element $(a',a'') \in \frac{1}{2}\zz^g \times \frac{1}{2} \zz^g$ is called a {\em characteristic} in degree~$g$.  We call two characteristics {\em distinct} when they are distinct modulo $\zz^{g} \times \zz^g$, and sometimes we refer to a characteristic as being the residue class of a characteristic modulo $\zz^{g} \times \zz^g$. For characteristics $a=(a',a'')$ and $b=(b',b'')$ we define the signs
\[ (a) = \exp(4\pi i {}^ta' a'') \, , \quad  (a,b)=\exp(4\pi i {}^ta'' b')  \, , \]
and
\[  ((a,b))=(a,b)(b,a)=\exp( 4\pi i ({}^ta' b'' - {}^tb' a'')) \, . \]
These signs only depend on the classes of $a, b$ modulo $\zz^{g} \times \zz^g$. 

The natural group operation on characteristics is usually written in a multiplicative manner to safe space.
Following this convention we have for example
\[ ((a,b)) = (a)(b)(ab) \, , \]
as is easily verified.
We put
\[ (((a,b,c)))=(a)(b)(c)(abc) =(bc)(ca)(ab)=((b,c))((c,a))((a,b)) \, . \]
A set of three distinct characteristics $a, b, c$ is called {\em syzygetic} if $(((a,b,c)))=+1$, and {\em azygetic} if $(((a,b,c)))=-1$. 
\begin{lem} \label{difference} Let $g \geq 1$ be an integer. Each non-zero characteristic in degree~$g$ can be written in exactly $2^{2g-2}$ ways as the difference of an odd and an even characteristic.
\end{lem}
This can be proven using affine geometry over $\mathbb{F}_2$. In Section~\ref{proof_difference} we give a proof based on the geometry of the moduli space of curves.

\subsection{Theta functions} \label{theta_fcns}
Let $\HH_g$ denote Siegel's upper half space of degree~$g$, that is the set of all complex symmetric $g$-by-$g$ matrices whose imaginary part is positive definite. Let $\tau \in \HH_g$. Then the abelian group $L_\tau =  \tau \zz^g + \zz^g $ is a lattice in $\cc^g$, and the complex torus $A_\tau = \cc^g/L_\tau$ has a natural structure of principally polarized complex abelian variety of dimension~$g$. 
\begin{definition} \label{thetas} (Theta functions) 
For $m=(m',m'') \in \zz^{g} \times \zz^g$ we put
\[  e_m(z,\tau) = \exp(-\pi i \,{}^t m'  \tau m' - 2\pi i \,{}^t m' z) \, , \quad z \in \cc^g \, . \]
Let $\ell \in \zz_{\geq 0}$.
 A holomorphic map $f \colon \cc^g \to \cc$ satisfying
\begin{equation} \label{func_eqn} f(z+ \tau m'+m'') =  e_m(z,\tau)^\ell \cdot f(z) \, , \quad m=(m', m'') \in \zz^{g}\times \zz^g \, , \, z \in \cc^g 
\end{equation}
is called a  {\em theta function of order} $\ell$ with respect to $L_\tau$.  
\end{definition}
The set $V_{\ell, \tau}$ of theta functions of order $\ell$ with respect to $L_\tau$ is a  complex vector space  of dimension $\ell^g$, consisting entirely of even functions (cf.\ \cite[Proposition~II.1.3]{tata1}).  When $f \in V_{\ell,\tau}$ is a non-zero theta function, we see from (\ref{func_eqn}) that the divisor of zeroes of $f$ on $\cc^g$ is $L_\tau$-periodic, and hence descends to give a well-defined effective divisor on the abelian variety $A_\tau = \cc^g/L_\tau$. We call $\tau \in \HH_g$ {\em indecomposable} if $A_\tau$ can not be written as the product of principally polarized abelian varieties of smaller dimension, and {\em decomposable} if $\tau$ is not indecomposable.

\subsection{Theta functions with characteristic} \label{theta_with_char}
\begin{definition} (Theta function with characteristic) For $ \tau \in \HH_g$ and $\a=(\a',\a'') \in \frac{1}{2}\zz^{g} \times \frac{1}{2}\zz^{g}$ we have the {\em theta function with characteristic} $\a$  given by the infinite series 
\begin{equation} \label{def_theta} \theta_\a(z, \tau) = \sum_{n \in \zz^g} \exp( \pi i  {}^t (n+\a')  \tau  (n+\a') + 2\pi i  {}^t (n+\a') (z+\a'')) \, , \quad z \in \cc^g \, . 
\end{equation}
\end{definition}
Each $\theta_\a$ defines a holomorphic function on $\cc^g$.
We just write $\theta(z,\tau)$ in the case $\a=0$.  
The function $\theta_\a=\theta_\a(z, \tau)$ is either odd or even as a function of~$z$, depending on whether $(\a)=-1$ or $(\a)=+1$.  
From the definition (\ref{def_theta}) one readily shows that for $m=(m', m'') \in \zz^{g} \times \zz^g$ and $\a=(\a',\a'') \in \frac{1}{2}\zz^{g} \times \frac{1}{2}\zz^{g}$ we have a functional equation
\begin{equation} \label{functional}
 \theta_\a(z+ \tau m' + m'',\tau) = \sqrt{((\a,m))} \cdot e_m(z,\tau) \cdot \theta_\a(z,\tau) \, , \quad z \in \cc^g \, ,
 \end{equation}
where $\sqrt{((\a,m))} $ is the sign $\exp(2\pi i ({}^t \a' m'' - {}^t m' \a'')) $. We see that the divisor of zeroes of the holomorphic function $\theta_\a$ is $L_\tau$-periodic, and hence descends to give a well-defined effective divisor on the abelian variety $A_\tau = \cc^g/L_\tau$. 
We put
\begin{equation} \label{exp_factor} \eta_\a(z,\tau) = \exp(- \pi i {}^t \a' \tau \a' - 2\pi i {}^t \a' (z+ \a'')) \, . 
\end{equation}
Let $\theta(z,\tau)$ be the theta function with zero characteristic. A small calculation yields that
\begin{equation} \label{theta_translate}
\theta(z+\tau \a' + \a'',\tau) = \theta_\a(z,\tau)   \cdot \eta_\a(z,\tau) \, .
\end{equation}
Equation (\ref{theta_translate}) shows that the divisors of the various $\theta_\a$ on $A_\tau$ are translates of one another by two-torsion points. We call a characteristic $a$ \emph{vanishing} (with respect to $\tau$) if $\theta_a(0,\tau)=0$.

The functional equation (\ref{functional}) shows that, somewhat confusingly, $\theta_\a$ is not a theta function in the sense of Definition \ref{thetas}, unless $\a=0$. The function $\theta$ is not identically zero and we see that $V_{1,\tau} = \cc \cdot \theta$. Each square of a $\theta_\a$ is a theta function of order two, however. Moreover we have that $\theta_\a^2= \theta_{\a+m}^2$ for $m \in \zz^{g} \times \zz^g$, in other words the theta functions $\theta_\a^2$ are independent of the choice of representative of $\a$ modulo $\zz^{g} \times \zz^g$. It can be shown that the squares $\theta_\a^2$, where $\a$ runs over all characteristics, span the space $V_{2,\tau}$ of theta functions of order two (cf.\ \cite[p.~618]{vdgvg}).  Frobenius' theta function, to be introduced in Section~\ref{Frobenius}, will be a designated element of the space $V_{2,\tau}$, for $\tau$ an indecomposable element of $\HH_3$.

Let $\varTheta_\tau$ denote the divisor of the theta function $\theta(z,\tau)$ with zero characteristic  on $A_\tau$, and let $\ell \in \zz_{\geq 0}$. Then pullback along the canonical projection $\cc^g \to A_\tau$ yields an isomorphism $\varGamma(A_\tau,\oo_{A_\tau}(\ell\,\varTheta_\tau)) \isom V_{\ell, \tau}$ of $\cc$-vector spaces. We have that $A_\tau$ is indecomposable if and only if the divisor $\varTheta_\tau$ of $A_\tau$ is irreducible.

\section{Moduli of abelian varieties} \label{modularity}
Let $\widetilde{q} \colon \uu_g \to \hh_g$ denote the {\em universal abelian variety} over $\hh_g$, by which we mean the family of complex tori over $\HH_g$ whose fiber at $\tau \in \hh_g$ is given by the abelian variety $A_\tau = \cc^g/L_\tau$. We note that $\uu_g$ can be realized as the quotient of the product space $\cc^g \times \hh_g$ by the action of $\zz^g \times \zz^g$ given by $(m',m'') \cdot (z,\tau) = (z+\tau m'+m'',\tau)$. Let $\Omega_{\uu_g/\hh_g}$ denote the sheaf of relative $1$-forms of $\widetilde{q}$. The {\em Hodge bundle} on $\HH_g$ is the pushforward sheaf $\widetilde{\ee}=\widetilde{q}_* \Omega_{\uu_g/\hh_g}$. This is a vector bundle of rank~$g$ on $\HH_g$. We denote its determinant by $\widetilde{\ll}$. The vector bundle $\widetilde{\ee}$ is globally trivialized by the frame $ ( \d z_1,\ldots,  \d z_g)$, and the determinant line bundle $\widetilde{\ll} = \bigwedge^g \widetilde{\ee}$ is globally trivialized by the frame $\d z_1 \wedge \ldots \wedge \d z_g$. 

We write elements of the symplectic group $\Sp(2g,\zz)$ as $\left(\begin{smallmatrix} A&B\\ C&D \end{smallmatrix}\right)$
where $A, B, C, D$ are $g \times g$ matrices. We have a left action of $\Sp(2g,\zz)$ on $\HH_g$ given by 
putting
\[  \left(\begin{smallmatrix} A&B\\ C&D \end{smallmatrix}\right) \cdot \tau =    (A\tau+B)(C\tau +D)^{-1} \, . \]
As is well-known the quotient $\aa_g = \Sp(2g,\zz) \setminus \hh_g$ can be identified with the moduli space of principally polarized abelian varieties.
We write $q \colon \xx_g \to \aa_g$ for the universal abelian variety over $\aa_g$. In this paper, all moduli spaces are viewed as orbifolds, or more accurately, as stacks.

More generally, let $\varGamma \subseteq \Sp(2g,\zz)$ be a finite index subgroup. Then $\aa_g^\varGamma = \varGamma \setminus \hh_g$ can be viewed as the moduli space of principally polarized complex abelian varieties of dimension~$g$ and level $\varGamma$. Let $q^\varGamma \colon \xx_g^\varGamma \to \aa_g^\varGamma$ denote the universal abelian variety over $\aa_g^\varGamma$, and let $\Omega_{\xx_g^\varGamma/\aa_g^\varGamma}$ denote the sheaf of relative $1$-forms of $q^\varGamma$. We have a Hodge bundle $\ee^\varGamma=q^\varGamma_* \Omega_{\xx^\varGamma_g/\aa^\varGamma_g}$ and associated determinant line bundle $\ll^\varGamma = \bigwedge^g \ee^\varGamma$ on $\aa_g^\varGamma$. Both $\ee^\varGamma$ and $\ll^\varGamma$ are obtained as quotients of their counterparts $\widetilde{\ee}$ and $\widetilde{\ll}$ on $\hh_g$ by a natural action of the group $\varGamma$. We just write $\ee$ and $\ll$ in the case of the full level subgroup $\varGamma=\Sp(2g,\zz)$. As any $\ee^\varGamma$ or $\ll^\varGamma$ can be obtained by pullback from $\aa_g$ we often leave out the superscript $\varGamma$ and just write $\ee$ or $\ll$ for the Hodge bundle and its determinant on $\aa_g^\varGamma$.

\subsection{Siegel modular and Siegel-Jacobi forms} \label{siegel-jacobi} 
Assume from now on that $g \geq 2$.
Let $k \in \zz$ and let $\varGamma \subseteq \Sp(2g,\zz)$ be a finite index subgroup. A {\em Siegel modular form of weight $k$ and level $\varGamma$} is a holomorphic section of the line bundle $\ll^{\otimes k}$ on $\aa_g^\varGamma$, or, equivalently, a
holomorphic function $f \colon \hh_g \to \cc$ such that for all $\left(\begin{smallmatrix} A&B\\ C&D \end{smallmatrix}\right) \in \varGamma$ the identity
\[ f( \left(\begin{smallmatrix} A&B\\ C&D \end{smallmatrix}\right) \cdot \tau) = \det(C\tau +D)^k \cdot f(\tau) \, , \quad \tau \in \hh_g  \]
is satisfied.
As above let $q^\varGamma \colon \xx_g^\varGamma \to \aa_g^\varGamma$ denote the universal abelian variety over $\aa_g^\varGamma$. Let $\xx_g^{\varGamma,\lor}$ denote the dual of the universal abelian variety over $\aa_g^\varGamma$, and let $\lambda^\varGamma \colon \xx_g^\varGamma \isom \xx_g^{\varGamma,\lor}$ be the universal principal polarization. Let $\pp^\varGamma$ denote the canonically rigidified Poincar\'e line bundle on $\xx_g^\varGamma \times_{\aa_g^\varGamma} \xx_g^{\varGamma,\lor}$, and write $\bb^\varGamma$ for the rigidified line bundle $(\id,\lambda^\varGamma)^*\pp^\varGamma$ on $\xx_g^\varGamma$. 
For each $\varGamma$ we can obtain $\bb^\varGamma$ by pullback from $\xx_g$ and as a result we will usually leave out the superscript $\varGamma$.  The restriction of $\bb$ to a given fiber of the structure map $q^\varGamma$ represents twice the cohomology class of the given principal polarization.

 Let $k, m \in \zz$. A {\em Siegel-Jacobi form of weight $k$, index $m$ and level $\varGamma$} is a holomorphic section of the line bundle $\ll^{\otimes k} \otimes \bb^{\otimes m}$ over $\xx_g^\varGamma$. When we only ask for a meromorphic section, then we rather speak about a {\em meromorphic Siegel-Jacobi form}.  We extend the left action of $\Sp(2g,\zz)$ on $\HH_g$ to a left action of $(\zz^g \times \zz^g ) \rtimes \Sp(2g,\zz)$ on $\cc^g \times \HH_g$ by putting
\[ \left(\begin{smallmatrix} A&B\\ C&D \end{smallmatrix}\right) \cdot (z,\tau) = ({}^t (C\tau+D)^{-1} z,     (A\tau+B)(C\tau +D)^{-1}) \, , \quad  \left(\begin{smallmatrix} A&B\\ C&D \end{smallmatrix}\right) \in \Sp(2g,\zz) \, , \]
\[ (n', n'') \cdot (z,\tau) = (z+\tau n' + n'',\tau) \, , \quad (n',n'') \in \zz^g \times \zz^g \, . \]
Using this action Siegel-Jacobi forms of weight $k$, index $m$ and level $\varGamma$ can be identified with functions $f(z,\tau)$ holomorphic on $\cc^g \times \HH_g$ such that
\[ f( \left(\begin{smallmatrix} A&B\\ C&D \end{smallmatrix}\right) \cdot (z,\tau)) = \det(C\tau +D)^k \cdot \exp(2\pi i m \,{}^t z \cdot (C\tau+D)^{-1} \cdot C \cdot z )\, f(z,\tau) \, , \]
\[ f( (n',n'') \cdot (z,\tau) ) =  e_n(z,\tau)^{2m} \cdot f(z,\tau) \, , \quad (n', n'') \in \zz^{g}\times \zz^g \, . \]
We see that when $f(z,\tau)$ is a Siegel-Jacobi form of index $m$ the restriction of $f(z,\tau)$ to $\cc^g \times \{\tau\}$ is a theta function of order $2m$ with respect to $L_\tau$. 

By a slight abuse of language, we call a holomorphic function $f(z,\tau)$ a Siegel-Jacobi form of weight $k$ and index $m$ if some power $f^{\otimes n}$ is of weight $kn$ and index $mn$. With this convention, we have for every characteristic $\a$ in degree~$g$ that the square $\theta_\a^2(z,\tau)$ of the theta function with characteristic~$a$ is a Siegel-Jacobi form of weight one and index one, cf.\ \cite[Section~II.5]{tata1}. 

\subsection{Moduli of curves} \label{moduli_curves} Let $g \geq 2$ be an integer. We denote by $\mm_g$ the moduli space of compact Riemann surfaces of genus~$g$, and by $p \colon \jj_g \to \mm_g$ the universal jacobian over $\mm_g$. Recall that we view moduli spaces as stacks. We have a Torelli map $t \colon \mm_g \to \aa_g$ that associates to every compact Riemann surface of genus~$g$ its jacobian. The universal jacobian $p \colon \jj_g \to \mm_g$ is obtained by pulling back the universal abelian variety $q \colon \xx_g \to \aa_g$ along the Torelli map. Let $\ll$ denote the determinant of the Hodge bundle on $\aa_g$, and $\bb$ the line bundle $(\id,\lambda)^*\pp$ on $\xx_g$ as in Section~\ref{siegel-jacobi}. By a slight abuse of notation we shall also denote by $\ll$ resp.\ $\bb$ the pullback of these line bundles to $\mm_g$ resp. $\jj_g$ along the Torelli map. 

\subsection{Differences of characteristics} \label{proof_difference}
Using the moduli space of curves one can give a nice proof of Lemma~\ref{difference}. The proof we present here is due to Emre Sert\"oz.
First of all we note that the statement of the lemma is equivalent with the following statement: let $X$ be a compact connected Riemann surface of genus~$g$. Then each non-zero two-torsion point of the jacobian of $X$  can be written in exactly $2^{2g-2}$ ways as the difference of an odd and an even spin bundle. Here by a {\em spin bundle} we mean a line bundle whose square is a canonical line bundle. By the {\em sign} of a spin bundle $\eta$ on a compact Riemann surface $X$ we mean the parity of the dimension $h^0(X,\eta)$ of the vector space of global sections of $\eta$. 
 
 Now the equivalent statement can be shown as follows. We may assume that $g \geq 2$. As above let $\mm_g$ denote the moduli space of compact Riemann surfaces of genus~$g$. 
Let $\mathcal{S}_g^{\pm}$ denote the moduli space of pairs of spin bundles $(X,\eta_1,\eta_2)$ with $\eta_1$ odd and $\eta_2$ even, and let $\jj_g[2]'$ denote the two-torsion substack of the universal jacobian over $\mm_g$, with the zero section excluded. The forgetful maps of both $\mathcal{S}_g^{\pm}$ and $\jj_g[2]'$ onto $\mm_g$ are finite \'etale maps. As the standard action of $\Sp(2g,\zz)$ on $\mathbb{F}_2^{2g} \setminus \{0\}$ is transitive, it follows from considerations as in \cite[Section~5.14]{dm} that the monodromy action on a fiber of $\jj_g[2]'$ over $\mm_g$  is transitive and hence that $\jj_g[2]'$ is connected. This allows us to deduce that the difference map $\mathcal{S}_g^{\pm} \to \jj_g[2]'$ is finite \'etale, in particular surjective with constant degree.  
There are $2^{g-1}(2^g-1) $ odd characteristics in degree~$g$, and $2^{g-1}(2^g+1)$ even ones. Hence
the degree of $\mathcal{S}_g^{\pm}$ over $\mm_g$ is $2^{g-1}(2^g-1) \times 2^{g-1}(2^g+1)$. As the degree of $\jj_g[2]'$ over $\mm_g$ is $2^{2g}-1$ it follows that the degree of the difference map $\mathcal{S}_g^{\pm} \to \jj_g[2]'$ is $2^{2g-2}$. This finishes the proof of Lemma~\ref{difference}.

\section{Frobenius' theta function} \label{Frobenius}

The purpose of this section is to define Frobenius' theta function, and to review some of its properties, following \cite{frob}. We start with a discussion of fundamental systems in degree three. 
Using the notion of a fundamental system we will then introduce the so-called ``reduced values'' in degree three. We characterize Frobenius' theta function in terms of these reduced values. We then continue to discuss the connection with the difference surface, explicit formulas, modular properties, and finish with a discussion of the hyperelliptic case and the decomposable case. 

\subsection{Fundamental systems} \label{fundamental}

 Let $\ff$ be a set consisting of eight distinct theta characteristics in degree three. We call $\ff$ a {\em fundamental system} of characteristics if each subset of $\ff$ consisting of three characteristics is azygetic (cf.\ Section~\ref{chars}).  
 
The following properties of fundamental systems are discussed in \cite[Section~2, pp.~44--46]{frob}. The translate of a fundamental system by a  theta characteristic is again a fundamental system. The set of translates of a given fundamental system consists of 64 distinct systems, and is called a {\em pencil} of fundamental systems. A fundamental system contains either three or seven odd characteristics. In a pencil of fundamental systems there are eight fundamental systems that contain seven odd characteristics, and 56 systems that contain three odd characteristics. 

The sum of all eight characteristics in a fundamental system is denoted by $2j$. We have that $2j$ is an element of $\zz^3 \times \zz^3$ and hence that $j$ is a characteristic.  The sum  of all odd characteristics in a fundamental system is denoted by $k$. We have that $k$ is an even characteristic, and that $k$ is an invariant modulo $\zz^3 \times \zz^3$ through a pencil. Each even characteristic is obtained exactly once from a pencil in this manner, and hence there exist precisely $36$ pencils of fundamental systems. 

Let $\ff$ be a fundamental system, and write as above $k$ for the sum of all odd characteristics of $\ff$.  When $\{\alpha,\beta\}$ runs over all 28 two-element subsets of $\ff$, the characteristics $k\alpha\beta$ yield all odd theta characteristics, each of them once. When $\{\alpha, \beta,\gamma,\delta\}$ runs over all 70 four-element subsets of $\ff$, the characteristics $k \alpha\beta\gamma\delta$ yield all even characteristics different from $k$, each of them twice. In fact if $\ff=\{\alpha,\beta,\gamma,\delta,\kappa,\lambda,\mu,\nu\}$ then $k\alpha\beta\gamma\delta \equiv k \kappa\lambda\mu\nu$ modulo $\zz^3 \times \zz^3$.

Let $\ff=\{\alpha,\beta,\gamma,\delta,\kappa,\lambda,\mu,\nu\}$ be a fundamental system, and set $\a=\alpha\beta\gamma\delta$.  Then the eight-element multi-set $\{\alpha \a,\beta \a, \gamma \a, \delta \a, \kappa,\lambda,\mu,\nu \}$ is again a fundamental system, and the sum of the odd characteristics in it is equal to $k\a$. In this way each given fundamental system $\ff$ can be canonically extended to a complete system of representatives from all $36$ pencils of fundamental systems. 

\subsection{Reduced values}
Let $\a=(a',a'')$ be a non-zero characteristic in degree three, and let $\tau \in \hh_3$. Let $k$ be a characteristic in degree three such that both $k$ and $k\a$ are even. 
Let $\ff$ be a fundamental system with sum of all odd characteristics equal to $k$ and with decomposition $\ff=\{\alpha,\beta,\gamma,\delta,\kappa,\lambda,\mu,\nu\}$ such that $\a=\alpha\beta\gamma\delta$. For the existence of such a fundamental system we refer to Section~\ref{fundamental}. 

Let $2j$ be the sum of all eight characteristics in $\ff$. We recall that $j$ is then a characteristic. We put (cf.\ \cite[eqn. (13) on p.~49]{frob})
\begin{equation} \label{h}  h_\a(\tau)= (j,\a) \prod_{\varepsilon \in \{\kappa,\lambda,\mu,\nu\}} \theta_{k\beta\gamma\delta\varepsilon}(0,\tau) \theta_{k\alpha\gamma\delta\varepsilon}(0,\tau)\theta_{k\alpha\beta\delta\varepsilon}(0,\tau) \theta_{k\alpha\beta\gamma\varepsilon}(0,\tau) \,  . 
\end{equation}
It follows from Section~\ref{fundamental} that the 16 characteristics occurring in the product are all even. By Lemma~\ref{difference} each non-zero characteristic in degree three can be written in exactly 16 ways as the difference of an even and an odd characteristic. We conclude that the 16 characteristics occurring in the product are exactly those 16 even characteristics that added to $\a$ yield an odd characteristic. 

A small verification shows that $h_\a(\tau)$ is independent of the choice of even characteristic $k$, of fundamental system $\ff$, and of four-element subset $\alpha,\beta,\gamma,\delta$ of $\ff$. Moreover $h_\a(\tau)$ is invariant under replacing $\a$ by another representative of its class modulo $\zz^3 \times \zz^3$. We conclude that $h_\a(\tau)$ is an invariant of the non-zero two-torsion point $\tau \a'+  \a'' \bmod L_\tau$ of the abelian threefold $A_ \tau=\cc^3/L_\tau$.  

We set $h_0(\tau)$ to be equal to zero. From (\ref{h}) we see that when viewed as a function of $\tau \in \HH_3$ each $h_\a(\tau)$ is a Siegel modular form of weight eight on $\HH_3$. We call $h_\a(\tau)$ the {\em reduced value} at $\tau \in \HH_3$ determined by the characteristic~$\a$.

\subsection{Frobenius' theta function} \label{def_frob_theta}
For $\tau \in \HH_3$ we denote by $\varGamma_{00,\tau} \subset V_{2,\tau}$  the linear subspace of $V_{2,\tau}$ consisting of functions with vanishing multiplicity $\geq 4$ at the origin. When $\tau$ is indecomposable, the space $\varGamma_{00,\tau}$ is one-dimensional, cf.\  \cite[Section~3, p.~36]{frob}, or \cite[Proposition~1.1]{vdgvg}.
The following result characterizes Frobenius' theta function $\phi=\phi(z,\tau)$ as an element of $V_{2,\tau}$. 
\begin{thm} \label{frob_mod_form} Assume that $ \tau \in \HH_3$ is an indecomposable matrix.  Then there exists a unique  element $\phi=\phi(z,\tau)$ in  $V_{2, \tau}$ such that $\phi$ lies in $\varGamma_{00,\tau}$, and such that
for all characteristics $\a  \in \frac{1}{2}\zz^3 \times \frac{1}{2}\zz^3$ the identity
\begin{equation} \label{reduced_value}  \phi( \tau \a' +\a'', \tau) = h_\a(\tau) \, \eta^2_\a(0,\tau) 
\end{equation}
holds. Here $h_a(\tau)$ denotes the reduced value (\ref{h}) determined by $a$, and $\eta_a(z,\tau)$ is the exponential factor (\ref{exp_factor}). The theta function $\phi$ generates $\varGamma_{00,\tau}$ as a $\cc$-vector space.
\end{thm}
\begin{proof} In \cite{frob} Frobenius constructs an element $\phi$ of $V_{2,\tau}$ that lies in $\varGamma_{00,\tau} $, and he shows that for that $\phi$ and for each $\a$ equation (\ref{reduced_value}) is satisfied, cf.\ \cite[eqn. (9) on p.~47]{frob}. This proves the existence claim. Next we remark that there is at most one vanishing even theta characteristic in degree three. This implies that there exist characteristics $\a$ such that $h_\a(\tau)$ does not vanish. Note that $\eta_\a(0,\tau)=  \exp(- \pi i \,{}^t \a' \tau \a'  - 2\pi i \, {}^t \a' \a'')$ does not vanish either. As $\varGamma_{00,\tau}$ is one-dimensional, this proves uniqueness. It is now also clear that $\phi$ generates $\varGamma_{00,\tau}$.
\end{proof}

\subsection{The difference surface} \label{difference_surface}
Let $X$ be a compact and connected Riemann surface of genus three, and let $J$ be its jacobian. We have a natural difference map $X \times X \to J$ given by $(x,y) \mapsto [x-y]$. The {\em difference surface} of $J$, notation $F_X$, is the pushforward of the fundamental class of $X \times X$ along the difference map. As the difference map is generically finite onto its image this gives a structure of Weil divisor, and hence of Cartier divisor, on the cycle $F_X$.  Note that actually the difference map is generically of degree two onto its image if $X$ is hyperelliptic, and generically an isomorphism onto its image if $X$ is not hyperelliptic.

Let $\tau \in \HH_3$. We note that the matrix $\tau$ is indecomposable if and only if the principally polarized abelian threefold $A_ \tau$ is the jacobian of a compact Riemann surface of genus three. 
The divisor of Frobenius' theta function has the following nice interpretation, cf.\ \cite[eqns.\ (4) and (5) on p.~37]{frob}, see also~\cite[Proposition~2.1]{vdgvg}. 
\begin{thm} \label{divisor_F} Let $\tau \in \HH_3$ be an indecomposable matrix, and assume that $A_\tau$ is the jacobian of the compact connected Riemann surface $X$ of genus three.  The effective Cartier divisor determined by Frobenius' theta function $\phi(z,\tau)$ on $A_ \tau$ is equal to the difference surface $F_X$ in $A_\tau$.
\end{thm}
One beautiful aspect of Frobenius' theta function $\phi(z,\tau)$ is that the projectivized tangent cone of $F_X$ at the origin of $A_\tau$ coincides with the canonical image of $X$ inside $\mathbb{P}^2$, cf. \cite[around eqn.\ (6) on p.~37]{frob} or \cite[Proposition~2.6]{vdgvg}. 

\subsection{Modular properties}  \label{mod_prop}
One may generate explicit formulas for Frobenius' theta function $\phi(z,\tau)$ by writing down a basis of $V_{2,\tau}$ and using (\ref{reduced_value}) for sufficiently many characteristics~$a$. For example one could consider the squares of a suitable set of theta functions with characteristic (cf.\ Section~\ref{theta_with_char}). 
This approach leads to the following explicit formula, cf.\ \cite[eqn. (15) on p.~49]{frob}. 
\begin{thm} \label{thm_explicit_frob} Let $\tau \in \hh_3$ be an indecomposable matrix.
Let $\ff$ be a fundamental system of characteristics in degree three.  Let $k$ be the sum of the odd characteristics in $\ff$. Let $b$ be any characteristic in degree three. Then the identity 
\begin{equation} \label{explicit_frob}
\theta_k(0,\tau)^2 \, \phi(z,\tau) = \sum_{\lambda \in \ff} (kb\lambda,b\lambda) \, h_{b\lambda}(\tau) \, \theta_{kb\lambda}(z,\tau)^2   
\end{equation}
holds in $ V_{2, \tau}$. 
\end{thm}
Formula~(\ref{explicit_frob}) gives insight in the modular properties of $\phi$.
Let $\mathcal{I}_3 \subset \HH_3$ denote the open subset consisting of indecomposable matrices. Note  that the set $\mathcal{I}_3$ is $\Sp(6,\zz)$-invariant.  Recall that each reduced value $h_a(\tau)$ is a Siegel modular form of weight eight. Inspection of (\ref{explicit_frob}) shows that Frobenius' theta function $\phi(z,\tau)$ transforms as a Siegel-Jacobi form of weight eight and index one over $\cc^3 \times \mathcal{I}_3$. As the locus of decomposable abelian threefolds is of codimension two in $\hh_3$, we may conclude the following.
\begin{thm}  \label{extends_as} Frobenius' theta function $\phi(z,\tau)$ extends as a Siegel-Jacobi form of weight eight and index one over  $\cc^3 \times \hh_3$. 
\end{thm}
In \cite{frob} many more explicit formulas for $\phi(z,\tau)$ are discussed. We especially would like to point to the discussion in \cite[Section~14]{frob}, where a connection is made with the Schottky-Igusa modular form in genus four. This circle of ideas is also the point of departure of the papers \cites{cfg, gsm, vdgvg}. The discussion in \cites{cfg, gsm, vdgvg}  shows that actually Frobenius' theta function $\phi$ is a Siegel-Jacobi form with respect to the {\em full} modular group $\Sp(6,\zz)$. 

Alternatively one can see this from the description of the zero locus of $\phi$ given by Theorem~\ref{divisor_F}. 
Let $\mm_3$ denote the moduli space of compact Riemann surfaces of genus three.
Let $p \colon \jj_3 \to \mm_3$ denote the universal jacobian over $\mm_3$, and let $p^1 \colon \CC_3 \to \mm_3$ denote the universal Riemann surface over $\mm_3$. We have a natural difference map $\delta \colon \CC_3 \times_{\mm_3} \CC_3 \to \jj_3$, which is proper. The {\em universal difference surface}, notation $F$, is  the pushforward of the fundamental class of $\CC_3 \times_{\mm_3} \CC_3 $ along the proper map $\delta$. The map $\delta$ is generically an isomorphism onto its image. We see that the cycle $F$ has codimension one on $\jj_3$. As $\jj_3$ is smooth, we have on $F$ a natural structure of relative effective Cartier divisor. 

The quotient space $\Sp(6,\zz) \setminus \mathcal{I}_3$ is naturally identified with $\mm_3$.  The quotient of $\cc^3 \times \mathcal{I}_3$ with respect to the natural $(\zz^3 \times \zz^3 ) \rtimes \Sp(6,\zz)$-action is identified with the universal jacobian $ \jj_3 $.  Let $\ll$ be the determinant of the Hodge bundle on $\mm_3$, and let $\bb$ be the canonical rigidified line bundle on $\jj_3$ as in Section~\ref{moduli_curves}. 
\begin{thm} \label{level_one} Frobenius' theta function $\phi(z,\tau)$ descends along the full modular $(\zz^3 \times \zz^3 ) \rtimes \Sp(6,\zz)$-action on $ \cc^3 \times \mathcal{I}_3$ to give a holomorphic section $\phi$ of the line bundle $\ll^{\otimes 8} \otimes \bb$ on $\jj_3$. The divisor of $\phi$ is the universal difference surface $F$.
\end{thm}
\begin{proof} Write $\mathcal{K}=\oo(F)$ for the line bundle determined by the relative effective Cartier divisor $F$ on $\jj_3$. By Theorem~\ref{divisor_F}, in each fiber of the map $p$ the surface $F$ is given by a theta function of order two. It follows that $\mathcal{K}$ and $\bb$ are fiberwise isomorphic. As the map $p$ is proper and as $\Pic(\mm_3)$ is generated by the class of $\ll$ (cf.\ \cite[Theorem~1]{ac}) it follows that there exists $e \in \zz$ such that $\mathcal{K} \cong \ll^{\otimes e} \otimes \bb$, and hence that there exists a Siegel-Jacobi form $\phi'$ of weight $e$, index one and full level in degree three such that $F = \divisor \phi'$. Combining with Theorem~\ref{extends_as} we see that the quotient $\phi'/\phi$ can be viewed as a meromorphic Siegel modular form of weight $e-8$ on $\mathcal{I}_3$ which  has no zeroes or poles on $\mathcal{I}_3$. Let $\varGamma$ be the level of $\phi$. Then $\varGamma$ determines a finite cover $\mm_3^\varGamma$ of $\mm_3$ to which $\phi'/\phi$ descends. As $\ll$ is ample on $\mm_3^\varGamma$ and as $\mm_3^\varGamma$ contains complete curves it follows that $e=8$ and $\phi'/\phi$ is a constant. We find that $\phi$ is of full level and that $F=\divisor \phi$ as claimed.
\end{proof}
Let $\a$ be a characteristic in  $\frac{1}{2}\zz^3 \times \frac{1}{2}\zz^3$. For future use we define the meromorphic function 
\begin{equation} \label{define_chi}  f_\a(z, \tau) = \phi(z,\tau)/\theta_\a(z, \tau)^2 
\end{equation}
on $\cc^3 \times \hh_3$. As $\theta_a^2$ is of weight one and index one (cf.\ Section~\ref{siegel-jacobi}) we see that $f_\a$ is a meromorphic Siegel-Jacobi form of weight seven and vanishing index. Let $\tau \in \HH_3$ be an indecomposable matrix. Then $f_\a(z, \tau)$ is an $L_\tau$-periodic function on $\cc^3$. It follows that $f_\a$ descends to give a meromorphic function on $A_ \tau$. Write $\varTheta_a$ for the divisor of $\theta_a$ and assume that $A_\tau$ is the jacobian of a compact Riemann surface $X$ of genus three. Then we see that $\divisor f_a $ equals the divisor $ F_X - 2\,\varTheta_a$.

\subsection{The hyperelliptic case} \label{hyp_case} Let $\tau \in \HH_3$, and assume that $A_ \tau$ is the jacobian of a {\em hyperelliptic} Riemann surface of genus three.  Let $k$ be the unique vanishing even theta characteristic for $ \tau$, by which we mean the unique even characteristic $k$ such that $\theta_k(0, \tau)=0$.  It is easily seen that the function $\theta_k(z, \tau)^2$ is a non-zero element  of $\varGamma_{00,\tau}$, and by Theorem~\ref{frob_mod_form} we find that $f_k(z, \tau)=\phi(z,\tau)/\theta_k(z, \tau)^2$ is a non-zero constant $\psi(\tau)$ independent of $z$. We would like to compute $\psi(\tau)$. 

Let $\a$ be a non-zero characteristic in degree three. Then by (\ref{theta_translate}) we note that $\theta_k(\tau \a' +\a'', \tau)$ vanishes if and only if $\theta_{k\a}(0, \tau)$ vanishes  if and only if $k\a$ is odd. We conclude by (\ref{reduced_value}) that $h_\a(\tau)$ vanishes if and only if $k\a$ is odd. 

Assume that $k\a$ is even. From (\ref{reduced_value}) and (\ref{theta_translate}) we obtain 
\begin{equation} \label{formula_chi} \psi(\tau) = \phi(\tau\a' +\a'',\tau)/\theta_k(\tau \a' +\a'',\tau)^2 = (k,\a)\,h_\a( \tau)/\theta_{k\a}(0, \tau)^2 \, , 
\end{equation}
cf.\ \cite[p.~50]{frob}.
We conclude that $\psi(\tau)$ transforms as a modular form of weight seven on the locus $D \subset \HH_3$ of hyperelliptic period matrices.

Following \cite{lock} we consider 
\begin{equation} \label{defxi} \xi( \tau) = \prod_{\a \colon (\a)=1, \a \neq k} \theta_\a(0, \tau)^8 \, , 
\end{equation}
the product of all non-vanishing even Thetanullwerte associated to $\tau$, raised to the eighth power.
As follows from \cite[Section~3]{lock}  the function $\xi=\xi(\tau)$ transforms as a modular form of weight~140 on $D$. 

The modular forms $\psi$ and $\xi$ are connected in the following manner. 
\begin{prop} \label{ups_and_xi} The equality $ \psi^{140}= \xi^7 $ holds on the hyperelliptic locus $D \subset \hh_3$.
\end{prop}
\begin{proof} We note that there are 35 non-zero characteristics $\a$ such that $k\a$ is even. From (\ref{h}) and (\ref{formula_chi}) we compute
\begin{equation} \begin{split} \psi(\tau)^{140} & = \prod_{\a \colon \a \neq 0, (k\a)=1} \theta_{k\a}(0,\tau)^{-8}\cdot h_\a(\tau)^4 \\
& = \prod_{\a \colon \a \neq 0, (k\a)=1} \left( \theta_{k\a}(0,\tau)^{-8} \prod_{b \colon (kb)=1, (k\a b)=-1 } \theta_{kb}(0,\tau)^4  \right) \\
& = \xi(\tau)^{-1} \cdot \prod_{b \colon (kb)=1, b \neq 0} \prod_{\a \colon (k\a)=1, (k\a b)=-1} \theta_{kb}(0,\tau)^4 \, . 
\end{split} 
\end{equation}
By Lemma~\ref{difference} each non-zero characteristic in degree three is written in exactly 16 ways as a difference of an odd and an even characteristic. Thus the second product on the last line runs over 16 values of $\a$. We deduce
\begin{equation}  \psi(\tau)^{140}  = \xi(\tau)^{-1} \cdot \prod_{b \colon (kb)=1, b \neq 0}  \theta_{kb}(0,\tau)^{64} 
 = \xi(\tau)^{-1} \cdot \xi(\tau)^8 = \xi(\tau)^7 \, , 
\end{equation}
which was to be shown.
\end{proof}

\subsection{The decomposable case}
In Theorem~\ref{level_one} we have seen that $\phi$ induces a holomorphic section of the line bundle $\ll^{\otimes 8} \otimes \bb$ over $\jj_3$.  It is natural to ask how $\phi$ behaves near the boundary of $\jj_3$. Theorem~\ref{vanishing_mult} below gives a partial answer to this question. Since we do not need the result in the rest of the paper, we only give a sketch of the proof.

For any $g \geq 2$ we let $\mm_g^\ct$ denote the moduli space of stable curves of genus~$g$ of {\em compact type}, i.e.\ of stable curves whose generalized jacobian is an abelian variety. Let $\varDelta \subset \mm_g^\ct$ denote the boundary of $\mm_g$ in $\mm_g^\ct$. Then $\varDelta$ is a divisor in $\mm_g^\ct$. 
\begin{thm} \label{vanishing_mult} Let $\pi \colon \overline{\jj_3}  \to \mm_3^\ct$ be the universal generalized jacobian over $\mm_3^\ct$ and write $E=\pi^{-1}\varDelta$. Frobenius' theta function $\phi$ extends as a holomorphic section of the line bundle $\ll^{\otimes 8} \otimes \bb$  over $ \overline{\jj_3}$. It has vanishing multiplicity two along~$E$.
\end{thm}
\begin{proof} Let  $t \colon \mm_3^{\ct} \to \aa_3$ be the extended Torelli map, given by associating to a stable curve of compact type its generalized jacobian. Then the universal generalized jacobian $\pi \colon \overline{\jj_3}  \to \mm_3^\ct$ is the pullback along $t$ of the universal abelian threefold over $\aa_3$. The first statement is then clear from Theorems~\ref{extends_as} and~\ref{level_one}. 

The vanishing multiplicity of $\phi$ along $E$ can be computed by looking, locally around a generic point of $\varDelta$, at the pullback of $\phi$ along a non-zero two-torsion section. By Theorem~\ref{frob_mod_form} we see that the required vanishing multiplicity is that of the reduced value $h_\a(\tau)$ along $\varDelta$, where $\a$ is any non-zero characteristic. We recall (cf.\  (\ref{h})) that $h_\a$ is, up to a sign, the product of $16$ Thetanullwerte, where the even characteristics occurring in the product are exactly the even characteristics that added to $\a$ yield an odd characteristic. 

Now among all 36 even theta characteristics in degree three, there are exactly six that split into an odd characteristic in degree two and an odd characteristic in degree one. The six associated Thetanullwerte are precisely the Thetanullwerte that vanish along $\varDelta$, and they all do so with multiplicity one (cf.\ \cite[p.~852]{ig}). A small verification shows that among the 16 even characteristics associated to a given $h_\a$, there are exactly two that split into an odd characteristic in degree two and an odd characteristic in degree one. Hence each $h_a(\tau)$ with $a$ non-zero vanishes with multiplicity two along $\varDelta$.
\end{proof}
\begin{remark} 
It follows directly from Theorem~\ref{vanishing_mult}  that the invariant $\log\|K\|$ from (\ref{def_jacobian_invariants}) behaves like $2 \log|t|$ near $\varDelta$ if $t$ is a local equation for $\varDelta$. The invariant $\log\|H\|$ on the other hand is locally bounded near every point of $\varDelta$. Using the formulae in Theorem~\ref{main} we recover in genus three the asymptotic results for $\varphi$ and $\delta$ near $\varDelta$ as given by \cite[Proposition~6.3]{wilms} in any genus.
\end{remark}

 \section{Arakelov geometry on moduli spaces of abelian varieties} \label{sec:ar_geom_A_g}
 
 The purpose of this section is to review some fundamentals from the Arakelov geometry of the moduli space $\aa_g$ of principally polarized complex abelian varieties of dimension~$g$. More precisely we will endow the line bundles $\ll$ and $\bb$ as encountered in Section~\ref{modularity} with certain canonical smooth hermitian metrics. This results into canonical normalized, real analytic versions of Siegel modular forms resp.\ of Siegel-Jacobi forms.  Let $g \geq 2$ be an integer. 
 
 \subsection{The Hodge metric} \label{hodge_metric}
 As in Section~\ref{modularity} let $\widetilde{q} \colon \uu_g \to \hh_g$ denote the universal abelian variety over $\hh_g$, let $\Omega_{\uu_g/\hh_g}$ denote the sheaf of relative $1$-forms of $\widetilde{q}$, let $\widetilde{\ee}=\widetilde{q}_* \Omega_{\uu_g/\hh_g}$ be the Hodge bundle on $\hh_g$, and $\widetilde{\ll}$ its determinant. We define the {\em Hodge metric} on $\widetilde{\ee}$ to be the smooth hermitian metric induced by the standard symplectic form on the natural variation of Hodge structures underlying the local system $R^1\widetilde{q}_* \zz_{\uu_g}$ on $\hh_g$. Recall that the line bundle $\widetilde{\ll}$ has a standard global frame $\d z_1 \wedge \ldots \wedge \d z_g $. The induced metric $\|\cdot\|_\Hdg$ on $\widetilde{\ll}$ can be given explicitly by 
\begin{equation} \label{norm_frame}  \|   \d z_1 \wedge \ldots \wedge \d z_g \|_\Hdg(\tau) = ( \det \Im \tau)^{1/2} \, , \quad \tau \in \hh_g \, . 
\end{equation}
We denote the curvature form of the hermitian line bundle $(\widetilde{\ll}, \|\cdot\|_\Hdg)$ by $\omega_\Hdg$. 

Let $\ll$ denote the determinant of the Hodge bundle on $\aa_g = \Sp(2g,\zz) \setminus \hh_g$. By construction the Hodge metric $\|\cdot\|_\Hdg$ on $\widetilde{\ll}$ is $\Sp(2g,\zz)$-invariant. We conclude that the Hodge metric descends to give a smooth hermitian metric on the line bundle $\ll$. The resulting metric and curvature form will also be denoted by $\|\cdot\|_\Hdg$ resp.\ $\omega_\Hdg$.

\subsection{Biextension metric} \label{biext_norm}

Let $q \colon \xx_g \to \aa_g$ denote the universal abelian variety over $\aa_g$, and let $\bb = (\id,\lambda)^*\pp$ on $\xx_g$ be the line bundle on $\xx_g$ derived from the universal Poincar\'e bundle, as in Section~\ref{siegel-jacobi}. Recall that the restriction of $\bb$ to a fiber of $q$ represents twice the given principal polarization.  Following \cite[Proposition~2.8]{bghdj} or \cite[Proposition~6.4]{hrar} the line bundle $\bb$ carries a unique smooth hermitian metric  which restricts to the trivial metric along the given trivialisation of $\bb$ along the zero section, and has translation-invariant curvature form in each fiber. 

The canonical metric on $\bb$  can be made explicit as follows. First of all we put
\[ \|P\|(z,\tau) = \exp(-\pi\, {}^t (\Im z) \cdot (\Im  \tau)^{-1} \cdot (\Im z)) \, , \quad (z,\tau) \in \cc^g \times \hh_g  \, . \]
Now for the moment let's fix $\tau \in \hh_g$. Let $\ell \in \zz_{\geq 0}$ and let $f \in V_{\ell,\tau}$ be a theta function of order~$\ell$ with respect to $L_\tau$ (cf.\ Section~\ref{theta_fcns}). The defining functional equation (\ref{func_eqn}) implies  that the real-valued function
\begin{equation} \label{def_norm_B}
 \|f\|_0 (z) = \|P \|(z,\tau)^\ell \cdot |f(z)| \, , \quad z \in \cc^g 
 \end{equation}
is invariant under translations by the lattice $L_\tau$. It follows that $\|f\|_0$ descends to give a function on the abelian variety $A_\tau=\cc^g/L_\tau$. Write $D=\divisor f$. The map $\|f \|_0$ defines a smooth hermitian metric $\| \cdot \|_0$ on $\oo_{A_\tau}(D)$ by setting $\|1\|_0 = \|f\|_0$ where $1$ is the canonical global section~$1$ of $\oo_{A_\tau}(D)$. 

Note that $\|f\|_0(0)= |f|(0)$.  It can be shown that  (\ref{def_norm_B}) with $\ell=2$  gives a globally defined smooth hermitian metric $\|\cdot\|_0$ on $\bb$ which restricts to the trivial metric along the given trivialisation along the zero section, and which has translation-invariant curvature form in each fiber. We have thus constructed the required metric on $\bb$ explicitly.

Denote by $2 \,\omega_0$  the curvature form of the metric $\|\cdot\|_0$.
By  (\ref{def_norm_B}) we can write, locally in coordinates $(z,\tau)$ coming from $\cc^g \times \hh_g$,
\begin{equation}  \label{explicit_omega_0}
  \omega_0 = \omega_0(z, \tau) = \frac{\deldelbar}{\pi i} \log \|P\|(z,\tau) =  i \, \deldelbar \, {}^t (\Im z)\cdot (\Im  \tau)^{-1} \cdot(\Im z) \, ,
\end{equation} 
see also \cite[Theorem~2.10]{bghdj}.

The following proposition lists a number of useful properties of the $2$-form $\omega_0$. 
\begin{prop} \label{vanishing_omega_power} The $2$-form $\omega_0$ vanishes along the zero section. On every fiber of the universal abelian variety $\xx_g \to \aa_g$ the $2g$-form $ \omega_0^g$ restricts to the Haar measure with total mass equal to $g!$. The $(2g+2)$-form $ \omega_0^{g+1}  $ vanishes identically on $\xx_g$. 
\end{prop}
\begin{proof} The first statement is immediate from the defining properties of $\|\cdot\|_0$. Next, by definition the restriction of $\omega_0$ to a given fiber is translation-invariant. It follows that the form $\omega_0^g$ restricts to a multiple of the Haar measure. As the restriction of $\omega_0$ to a  fiber is a de Rham representative of the given principal polarization, and as the degree of a divisor representing a principal polarization is equal to $g!$, we find the second statement. Finally, let $\sqrt{\Im  \tau}$ denote the unique positive definite square root of $\Im  \tau$. Then $\sqrt{\Im \tau}$ defines a real analytic function on $\hh_g$. From (\ref{explicit_omega_0}) we see that the coordinate transformation $u = (\sqrt{\Im  \tau})^{-1}z$ allows us to write
$ \omega_0(u, \tau)=  i \, \deldelbar \, {}^t (\Im u)  \cdot (\Im u) = \frac{i}{2} \sum_{j=1}^g \d  u_j \wedge \d  \overline{u_j}$.
In particular we can make $\omega_0$ to depend on only $g$ holomorphic coordinates. This gives the vanishing of $\omega_0^{g+1}$. 
\end{proof}

\subsection{Normalized Siegel-Jacobi forms} \label{normalization} 
Let $k, m \in \zz$. From the metric $\|\cdot\|_0$ on $\bb$ and the Hodge metric $\|\cdot\|_{\Hdg}$ on $\ll$ we obtain canonically induced hermitian metrics $\|\cdot\|$ on the line bundle $\ll^{\otimes k} \otimes \bb^{\otimes m}$ on $\xx_g$. 
In particular, if $f(z,\tau)$ is a Siegel-Jacobi form of weight~$k$ and index~$m$ with respect to some finite index subgroup $\varGamma$  of $\Sp(2g,\zz)$ (cf.\ Section~\ref{modularity}) we obtain a normalized version $\|f\|(z,\tau)$ of $f(z,\tau)$ by taking the norm of $f(z,\tau)$ in the metric $\|\cdot\|$. Explicitly we have
\[ \begin{split} \|f\|(z,\tau) & =  \|f\|_0(z,\tau) \cdot  \|   \d z_1 \wedge \ldots \wedge \d z_g \|^k_\Hdg(\tau) \\ 
& =   (\det \Im \tau)^{k/2}\cdot  \|f\|_0(z,\tau)  \, , \quad z \in \cc^g \, , \tau \in \HH_g \, ,
\end{split} \]
where $\|f\|_0(z,\tau)$ is defined as in (\ref{def_norm_B}) with $\ell=2m$. Here we have used (\ref{norm_frame}).
We have that $\|f\|(z,\tau)$ is $(\zz^g \times \zz^g) \rtimes \varGamma$-invariant, and we obtain a well-defined real valued function $\|f\|$ on $\xx_g^\varGamma$. 

\subsection{Normalized theta functions with characteristics}
A particular case is furnished by the theta functions with characteristics (cf.\ Section~\ref{theta_with_char}).
Let $\a \in \frac{1}{2}\zz^g \times \frac{1}{2}\zz^g$ be any characteristic in degree~$g$. From Section~\ref{siegel-jacobi} we recall that $\theta_\a^2(z,\tau)$ is a Siegel-Jacobi form of weight one and index one. We define $\|\theta_\a\| (z,\tau)$ to be the square root of the normalized theta function $ \|\theta_\a^2\|(z,\tau)$. From (\ref{theta_translate}) it is straightforward to verify that one has an equality
\begin{equation} \label{invariant_translate}
 \|\theta_\a\|(z,\tau) = \|\theta\|(z+\tau \a'+\a'',\tau) \, , \quad z \in \cc^g \, , 
\end{equation}
which shows that  the various functions $\|\theta_\a\|([z],\tau)$ on $A_\tau$ are translates of one another by two-torsion points. Here $\theta$ is the theta function with zero characteristic in degree~$g$.

Fix again a matrix $\tau \in \hh_g$. One may view $\|\theta\|$ as giving a smooth hermitian metric on $\oo_{A_\tau}(\varTheta_\tau)$ by setting $\|1\|=\|\theta\|$ where $1$ is the canonical global section of $\oo_{A_\tau}(\varTheta_\tau)$. 
Let $\mu_\tau$ denote the Haar measure on $A_\tau$ normalized to give $A_\tau$ total mass equal to one.  Following \cite[p.~401]{fa} or \cite[Section~2.1]{wilms}, the metric $\|\cdot\|$ on $\oo_{A_\tau}(\varTheta_\tau)$ is characterized among all smooth metrics on $\oo_{A_\tau}(\varTheta_\tau)$ by the following two properties: (a) the associated curvature form is translation-invariant, and (b) the function $\|1\|^2$ integrates to $2^{-g/2}$ against~$\mu_\tau$.

\begin{definition} Let $\tau \in \hh_g$ and let $\a \in \frac{1}{2}\zz^g \times \frac{1}{2}\zz^g$ be a characteristic in degree~$g$. Following \cite{bost_cras, bmmb} we define 
\begin{equation} \label{defH}
 \log \|H \|(A_\tau) = \int_{A_\tau} \log \|\theta_\a \| \, \mu_\tau \, , 
\end{equation} 
where as above $\mu_\tau$ is the Haar measure on $A_\tau$ normalized to give $A_\tau$ total mass equal to one.
As $\theta_a$ is not identically zero we see that $\log \|H\|(A_\tau)$ is well-defined as a real number.
By (\ref{invariant_translate})  the real number $\log \|H\|(A_\tau)$ is independent of the choice of $\a$.
As it turns out, the group $\Sp(2g,\zz)$ acts on the set of functions $\|\theta_\a\|$. We deduce that $\log \|H\|$ is an invariant of the isomorphism class of $A_\tau$ as a principally polarized complex abelian variety, and hence we have a natural induced real-valued map $\log\|H\| \colon \aa_g \to \rr$. The invariant $\log\|H\| $ coincides with the invariant called $H$ in \cite[Section~2.1]{wilms}.
\end{definition}
 
\section{Arakelov geometry on moduli spaces of curves} \label{arakelov}

 In our proof of Theorem~\ref{main}  we will need a few identities involving canonical $2$-forms on the moduli space of compact Riemann surfaces, and related moduli spaces, in the spirit of \cites{djtorus, kaw, kawhandbook, wilms}. The purpose of this section is to display these identities, referring to {\em loc.\ cit}.\ for further details and proofs.
 
Let $g \geq 2$ be an integer. As before we denote by $\mm_g$ the moduli space of compact Riemann surfaces of genus ~$g$. Let $p^1 \colon \CC_g \to \mm_g$ be the universal Riemann surface over $\mm_g$, and let $p^2 \colon \CC_g^2 \to \mm_g$ be the universal self-product of a Riemann surface over $\mm_g$. Thus $\CC_g$ parametrizes pairs $(X,x)$ where $X \in \mm_g$ and $x \in X$, and $\CC_g^2$ parametrizes triples $(X,x,y)$ where $X \in \mm_g$ and $x, y \in X$. Let $T_{\CC_g/\mm_g}$ denote the relative tangent bundle of $p \colon \CC_g \to \mm_g$, let $\varDelta$ denote the diagonal divisor of $\CC_g^2$, and let $\oo(\varDelta)$ denote the associated line bundle on $\CC_g^2$. 

We obtain a canonical smooth hermitian metric $\|\cdot\|$ on $\oo(\varDelta)$ by setting $\log \|1\|(X,x,y) = g_X(x,y)$ for $x \neq y$ with $g_X$ the {\em Arakelov-Green's function}  \cite{ar, fa} on $X$. Here $1$ denotes the canonical holomorphic section of $\oo(\varDelta)$. The Poincar\'e residue gives a canonical isomorphism of line bundles $T_{\CC_g/\mm_g} \isom \varDelta^* \oo(\varDelta)$ on $\CC_g$. By this isomorphism the metric $\|\cdot\|$ induces a canonical residual metric $\|\cdot\|_\Ar$ on $T_{\CC_g/\mm_g}$, called the {\em Arakelov metric}.
 
 We let $e^A$ denote the curvature form of the smooth hermitian line bundle $T_{\CC_g/\mm_g}$ over $\CC_g$, and we let $h$ denote the curvature form of the smooth hermitian line bundle $\oo(\varDelta)$ on $\CC_g^2$. We set $e_1^A = \int_{p^1} (e^A)^2$. This is a $2$-form on $\mm_g$.
 We denote by $p_i \colon \CC_g^2 \to \CC_g$ for $i=1, 2$ the projections on the first and second coordinate, respectively. By \cite[Proposition~6.2]{djtorus} we have for $i=1, 2$ an equality
\begin{equation} \label{self_inters}
 \int_{p_i} h^2 = e^A  
\end{equation}
of $2$-forms on $\CC_g$. From (\ref{self_inters}) and the projection formula for fiber integration one readily derives for each $i,j =1,2$ the identities
\begin{equation} \label{forms-I}
 \int_{p^2} h^2 \, p_i^*e^A = e_1^A \, , \quad \int_{p^2} h \, p_i^*e^A \, p_j^*e^A = e_1^A \, , \quad  \int_{p^2} (p_i^* e^A)^3 =0 \, , 
\end{equation}
and for $i \neq j$,
\begin{equation} \label{forms-II}
  \int_{p^2} p_i^*e^A \, (p_j^*e^A)^2 = (2-2g)\, e_1^A \, .  
\end{equation}
Write $e_1^K$ for the $2$-form $\int_{p^2} h^3$ on $\mm_g$. Let $\varphi \colon \mm_g \to \rr$ be the Kawazumi-Zhang invariant. By \cite[Proposition~2.5.3]{zhgs} the function $\varphi$ can be obtained as the fiber integral
\begin{equation}
 \varphi = \int_{p^2} \log \|1\| \, h^2 \, , 
\end{equation}
where we recall that $\log \|1\|(X,x,y) = g_X(x,y)$ is the Arakelov-Green's function. This expression for $\varphi$ readily gives the identity
\begin{equation} \label{variation_varphi}
\frac{\deldelbar}{\pi i} \varphi = e_1^K- e_1^A 
\end{equation}
of $2$-forms on $\mm_g$, cf.\ \cite[Proposition~5.3]{djtorus}.

As before let $q \colon \xx_g \to \aa_g$ denote the universal abelian variety over $\aa_g$. Denote by $2 \,\omega_0$  the curvature form of the canonical metric $\|\cdot\|_0$ on the line bundle $\bb$ on $\xx_g$ as in Section~\ref{biext_norm}.  Let $\omega_\Hdg$ be the curvature form of the Hodge metric on the determinant of the Hodge bundle $\ll$ over $\aa_g$ as in Section~\ref{hodge_metric}. Let $p \colon \jj_g \to \mm_g$ denote the universal jacobian over $\mm_g$. 
The forms $\omega_0$ resp.\ $\omega_{\Hdg}$ pull back along the Torelli map to $\jj_g$ resp.\ $\mm_g$. We denote the resulting forms by the same symbols. 

Let $\delta \colon \CC_g^2 \to \jj_g$ be the difference map that sends a triple $(X,x,y)$ to the pair $(X,[x-y])$, where $[x-y]$ is the class of the divisor $x-y$ in the jacobian of $X$.
By identity (K3) in \cite[Theorem~1.4]{djtorus} we have
\begin{equation} \label{delta_upper_star}  2 \,  \delta^* \omega_0 = 2\,h - p_1^*e^A - p_2^*e^A  \, . 
\end{equation}
Via the Torelli map we can view the function $\log \|H\|$ discussed at the end of Section~\ref{biext_norm} as a function on $\mm_g$. In \cite[eqn. (5.19)]{wilms} we find  the equality
\begin{equation} \label{variation_theta_integral}
\frac{\deldelbar}{\pi i} \log \|H\| = \frac{1}{2}  \, \omega_{\Hdg} - \frac{1}{8} \, e_1^A + \frac{1}{12} \, e_1^K 
\end{equation}
of $2$-forms on $\mm_g$.

\section{Proof of Theorem~\ref{main}} \label{sec:normalized}

In this section we specialize to the case where the genus is three. 

\subsection{Normalized Frobenius theta function} Denote by $p \colon \jj_3 \to \mm_3$ the universal jacobian over $\mm_3$. As we have seen in Theorem~\ref{level_one} Frobenius' theta function $\phi(z,\tau)$  determines a holomorphic section of the line bundle $\ll^{\otimes 8} \otimes \bb$ on $\jj_3$. Its normalized version $\|\phi\| = \|\phi\|(z,\tau)$ (cf.\ Section~\ref{normalization}) is $(\zz^3 \times \zz^3) \rtimes\Sp(6,\zz)$-invariant and hence descends to give a well-defined real valued function on $\jj_3$. The zero locus of $\|\phi\| $ is the universal difference surface $F$.
\begin{definition} Let $\tau \in \hh_3$.
Assume that $A_\tau$ is the jacobian of a compact Riemann surface $X$ of genus three, and let $\mu_\tau$ denote the Haar measure on $A_\tau$ giving $A_\tau$ unit volume. We set
\begin{equation} \label{defK}
 \log \|K\|(X)= \int_{A_\tau} \log \|\phi\| \, \mu_\tau \, . 
\end{equation}
As $\|\phi\|$ does not vanish identically on $A_\tau$ this indeed defines a real-valued invariant of the Riemann surface $X$. We may view $\log \|K\|$ as a function on $\mm_3$. 
\end{definition}

We have the following counterpart to (\ref{variation_theta_integral}) for the invariant $\log \|K\|$ in genus three.
\begin{thm}  Let $\log \|K\|$ be the invariant given in (\ref{defK}). We have an equality
\begin{equation} \label{variation_phi_integral}
\frac{\deldelbar}{\pi i} \log \|K\|=  8 \, \omega_{\Hdg} - \frac{1}{2}\, e_1^A - \frac{1}{6} \, e_1^K   
\end{equation}
of $2$-forms on $\mm_3$.
\end{thm}
\begin{proof}  Let $F$ be the universal difference surface over $\mm_3$. We recall from Section~\ref{mod_prop} that $F$ is an effective relative Cartier divisor over $\mm_3$. Note that the curvature form of the smooth hermitian line bundle $\ll^{\otimes 8} \otimes \bb$ is equal to $2\, \omega_0 + 8 \, \omega_{\Hdg}$. The Poincar\'e-Lelong formula therefore gives an equality of currents   
\[  \frac{\deldelbar}{\pi i} \log  \|\phi\| = 2\, \omega_0 + 8 \, \omega_{\Hdg} - \delta_{F} \]
on $\jj_3$.  By Proposition~\ref{vanishing_omega_power}  the differential form $ \frac{1}{6} \omega_0^3 $ restricts to the Haar measure of unit mass on each fiber of $\jj_3$ over $\mm_3$, and the differential form $\omega_0^4$ vanishes. 
Starting from (\ref{defK}) we deduce, using that $\omega_0^3$ is $\partial$- and $\overline{\partial}$-closed, 
\begin{equation} \label{useful} \begin{split} \frac{\deldelbar}{\pi i} \log \|K\| 
& =  \frac{1}{6}  \frac{\deldelbar}{\pi i}  \int_{p} \log  \|\phi\| \, \omega_0^3   \\
& =  \frac{1}{6}  \int_{p} \frac{\deldelbar}{\pi i} \log  \|\phi\| \, \omega_0^3   \\
& = \frac{1}{6} \int_{p} \left(  2 \,\omega_0 + 8 \, \omega_{\Hdg} - \delta_F   \right) \, \omega_0^3 \\
& = 8 \,\omega_{\Hdg} - \frac{1}{6}\int_{p|_F} \omega_0^3 \, . 
\end{split} 
\end{equation} 
Let $p^2 \colon \CC_3^2 \to \mm_3$ denote the projection map. The identities in (\ref{forms-I}), (\ref{forms-II}) and (\ref{delta_upper_star}) and the fact that the difference map $\delta$ is generically an isomorphism onto its image allow us to compute
\begin{equation} \label{fiber_int_overZ} \begin{split} \int_{p|_F} \omega_0^3 & = \frac{1}{8}  \int_{p^2} \left( 2h - p_1^*e^A - p_2^*e^A \right)^3 \\
& = \frac{1}{8}  \int_{p^2} \left( 8h^3 -12h^2(p_1^*e^A +p_2^*e^A) + 6h(p_1^*e^A+p_2^*e^A)^2 - (p_1^*e^A+p_2^*e^A )^3 \right) \\
& =  3\,e_1^A + e_1^K \, .
\end{split} 
\end{equation}
Combining (\ref{useful}) and (\ref{fiber_int_overZ}) we obtain (\ref{variation_phi_integral}).
\end{proof}

\subsection{Proof of the main result} \label{proof_main}

We can now prove Theorem \ref{main}.
By Wilms' identity (\ref{wilms})  it suffices to show (\ref{explicit_phi}).  First of all, a combination of (\ref{variation_varphi}),  (\ref{variation_theta_integral}) and (\ref{variation_phi_integral}) yields that $\deldelbar$ of left and right hand side of (\ref{explicit_phi}) are equal as $2$-forms on $\mm_3$.  As every pluriharmonic function on $\mm_3$ is constant, cf.\ the proof of \cite[Theorem~5.4.1]{wilms}, we will be done once we show that  formula (\ref{explicit_phi}) is correct for $X$ a hyperelliptic Riemann surface of genus three.
 
Let  $\tau \in \hh_3$ and assume that $A_\tau$ is isomorphic as a principally polarized abelian threefold to the jacobian $J$ of a hyperelliptic Riemann surface $X$.  Let $\xi(\tau)$ be given by (\ref{defxi}) and set $ \|\xi\|( \tau) = (\det \Im  \tau)^{70} |\xi( \tau)| $. By \cite[Corollary~4.10]{wilms} we have
\begin{equation} \label{first_step}
\varphi(X) = - \frac{1}{30} \log \|\xi\|(\tau) + \frac{28}{3} \log \|H\|(X) + 8 \log 2 \, . 
\end{equation}
Let $k$ be the unique vanishing even characteristic associated to $\tau$ and consider $\psi(\tau)=\phi(z,\tau)/\theta_k(z,\tau)^2$  as in Section~\ref{hyp_case}. We set, for generic $z \in \cc^3$,
\begin{equation} \label{norm_psi} \begin{split} \|\psi\|(\tau) & = \|\phi\|(z,\tau)/\|\theta_k\|(z,\tau)^2 \\ 
& = (\det \Im  \tau)^{7/2} |\phi(z,\tau)/\theta_k(z,\tau)^2| \\ 
& = (\det \Im  \tau)^{7/2} |\psi( \tau)| \, . \end{split} 
\end{equation}
Combining Proposition~\ref{ups_and_xi}, (\ref{defH}) and (\ref{norm_psi}) we obtain that 
\begin{equation} \label{hyp_xi}  -\log \|\xi\|( \tau)  = -20 \log \|\psi\|(\tau) =  -20 \log\|K\|(X) + 40 \log \|H\|(X)  \, . 
\end{equation}
Combining (\ref{first_step}) and (\ref{hyp_xi})  we obtain (\ref{explicit_phi}) for the hyperelliptic Riemann surface~$X$.

\section{A formula of Bloch, Hain and Bost revisited} \label{bloch}

Let $X$ be a compact Riemann surface of positive genus. For $n \in \zz$ we denote by $\Pic^n X$ the moduli space of linear equivalence classes of divisors of degree~$n$ on $X$. In particular, the moduli space $\Pic^0 X$ identifies with the jacobian of $X$.

For $\alpha \in \Pic^1 X$ we write $X_\alpha$ for the image of $X$ in $\Pic^0 X$ under the Abel-Jacobi map $x \mapsto [x-\alpha]$. Further we write $\varSigma_\alpha $ for the {\em Ceresa cycle} $ X_\alpha - [-1]_*X_\alpha$ in $\Pic^0 X$. Let $\alpha, \beta \in \Pic^1 X$. If the genus of $X$ is equal to three, and if the supports of the Ceresa cycles $\varSigma_\alpha$ and $\varSigma_\beta$ in $\Pic^0 X$ are disjoint, one has a natural associated {\em archimedean height pairing} $\pair{\varSigma_\alpha,\varSigma_{\beta} }_\infty \in \rr$. 

Around thirty years ago, Bloch asked to compute the archimedean height pairing $\pair{\varSigma_\alpha,\varSigma_{\beta} }_\infty$ in the case where $4\alpha$ is a canonical divisor class. Soon after Bloch had asked his question he and Hain and independently also Bost were able to carry out this computation. The results are contained in the papers \cite{bostduke} and \cite{hain}, which were published simultaneously. The approaches in the two papers \cite{bostduke} and \cite{hain} are however rather different. The aim of this final section is to bridge the two approaches, by making reference to Theorem~\ref{explicit_beta} and to our study of Frobenius' theta function $\phi$ as a Siegel-Jacobi form (cf.\ Theorem~\ref{level_one}). As this section contains no new results, we will be brief at most points in our discussion and leave the details to the interested reader.

\subsection{Hain-Reed invariant} \label{hainreed_bis}

We start by briefly recalling how the Hain-Reed invariant $\beta$ is defined, referring to the papers \cites{hrar, hr, hain_normal, dj_normal} for details on the various constructions. 

Let $g \geq 3$ be an integer. The first homology $H$ of a compact Riemann surface $X$ of genus~ $g$ defines a natural variation of polarized Hodge structures $\mathbb{H}$ on $\mm_g$. The intersection form on $H$ gives rise to an embedding $\mathbb{H} \to \bigwedge^3 \mathbb{H}$, and it can be shown  that the cokernel $\mathbb{L} = \bigwedge^3 \mathbb{H}/\mathbb{H} $ is again naturally a variation of polarized Hodge structures on $\mm_g$. 

Let $\jj(\mathbb{L}) \to \mm_g$ be the family of Griffiths intermediate jacobians associated to the variation $\mathbb{L}$. Taking the Abel-Jacobi class of a Ceresa cycle in the jacobian of $X$ defines a normal function section $\nu \colon \mm_g \to \jj(\mathbb{L})$ of $\jj(\mathbb{L}) \to \mm_g$. Following \cite{hrar} we may now consider a holomorphic line bundle $\nn $ on $\mm_g$ obtained by pulling back, along the normal function section $\nu$, a standard line bundle $\hat{\pp}$ on $\jj(\mathbb{L})$ obtained from the Poincar\'e bundle associated to the torus fibration $\jj(\mathbb{L})$ and the given polarization of $\jj(\mathbb{L})$. 

Let $\ll$ denote the determinant of the Hodge bundle on $\mm_g$ (cf. Section~\ref{moduli_curves}). It follows from a result of Morita \cite[(5.8)]{mo}, see also \cite[Theorem~7]{hr}, that the line bundle $\nn=\nu^*\hat{\pp}$ on $\mm_g$ is isomorphic with $\ll^{\otimes (8g+4)}$. We call a {\em Hain-Reed isomorphism} any choice of isomorphism $\ll^{\otimes (8g+4)} \isom \nn$. As every invertible holomorphic function on $\mm_g$ is constant we see that a Hain-Reed isomorphism is unique up to multiplication by constants. 

As is explained in {\em loc.\ cit.}, the line bundle $\hat{\pp}$ carries a canonical smooth hermitian metric, and by pullback along $\nu$ this defines a natural smooth hermitian metric on $\nn$. On the other hand, we can equip the line bundle $\ll$ with its Hodge metric, cf.\ Section~\ref{hodge_metric}. Taking the logarithm of the norm, with respect to the two given metrics on $\ll$ and $\nn$, of any Hain-Reed isomorphism $\ll^{\otimes (8g+4)} \isom \nn$ yields a well-defined class $\beta \in C^0(\mm_g,\rr)/\rr$.
This class $\beta$ is called the {\em Hain-Reed invariant} in genus~$g$. 

We recall that Theorem~\ref{explicit_beta} expresses the invariant $\beta$ in the case $g=3$ in terms of the invariants $\log \|K\|$ and $\log\|H\|$. We will apply Theorem~\ref{explicit_beta} in Section~\ref{applying}. In Section~\ref{connection} we will see an explicit construction of a Hain-Reed isomorphism $\ll^{\otimes 28} \isom \nn$ over $\mm_3$ using Frobenius' theta function $\phi$.

\subsection{Ceresa cycles and their height pairing in genus three}

Let $X$ be a compact Riemann surface of genus three. Let $\varTheta \subset \Pic^2 X$ be the locus of effective divisor classes, and let $F_X \subset \Pic^0 X$ be the difference surface as in Section~\ref{difference_surface}. We recall that both $\varTheta$ and $F_X$ are effective Cartier divisors. 

Let $\alpha, \beta \in \Pic^1 X$ and assume that $4\alpha$ is canonical. Let $\varTheta_{2\alpha} \subset \Pic^0 X$ be the translate of $\varTheta  \subset \Pic^2 X$ by the semi-canonical divisor $2\alpha$. It is not hard to see that the supports of $\varSigma_\alpha$ and $\varSigma_{\beta} $ are disjoint whenever $\alpha+\beta \notin |\varTheta|$ and $\alpha-\beta \notin |F_X|$, or equivalently whenever $\alpha-\beta \notin |\varTheta_{2\alpha}| \cup |F_X|$. 

Following Bloch's original idea \cite{hain} we consider the moduli space $\widetilde{\mm}$ of pairs $(X,\alpha)$ where $X$ is a compact Riemann surface of genus three and $\alpha \in \Pic^1 X$ is a divisor class such that $4\alpha$ is canonical. 
Let $\widetilde{p} \colon \widetilde{\jj} \to \widetilde{\mm}$ denote the universal jacobian over $\widetilde{\mm}$. We may view $\widetilde{\jj}$ as the moduli space of triples $(X,\alpha,D)$ where $(X,\alpha) \in \widetilde{\mm}$ and  $D$ is an element of $\Pic^0 X$. We naturally have a universal symmetric theta divisor $\varTheta_{0}$ and a universal difference surface $F$ on $\widetilde{\jj}$, both of which are effective relative Cartier divisors.  

Let $\mm$ be the moduli space of compact Riemann surfaces of genus three.
We have a natural forgetful map $\widetilde{\mm} \to \mm$.  Let $\ll$ denote the determinant of the Hodge bundle on $\mm$, endowed with the Hodge metric, and $\nn$ the Hain-Reed line bundle on $\mm$ as discussed in Section~\ref{hainreed_bis}, endowed with its canonical hermitian metric. By a slight abuse of notation we shall also denote by $\ll$ resp.\ $\nn$ the pullbacks of the hermitian line bundles $\ll$ resp.\ $\nn$ to $\widetilde{\mm}$ and to $\widetilde{\jj}$. Let $\jj_1 \to \mm$ be the family of {\em intermediate jacobians} of the universal jacobian threefold $p \colon \jj \to \mm$, and let $\mathcal{P}_1 \to \jj_1 \times_{\mm} \jj_1$ be the canonical biextension line bundle over $\jj_1 \times_{\mm} \jj_1$, endowed with its canonical hermitian metric, as in \cite[Section~1.1]{hain}. 

We denote by $\sigma \colon \widetilde{\jj} \to \jj_1 \times_\mm \jj_1$ the map over $\mm$ given by sending a triple $(X,\alpha,D)$ to the pair consisting of the Abel-Jacobi classes of the Ceresa cycles $\varSigma_\alpha$ resp.\ $\varSigma_{\alpha D}$ in the intermediate jacobian $J_1(X)$. Write $\bb_\sigma$ for the hermitian line bundle $ \sigma^*\mathcal{P}_1$ on $\widetilde{\jj}$. We denote its canonical metric by $\|\cdot\|_{\mathrm{biext}}$. 
 The line bundle $\bb_\sigma$ admits a canonical meromorphic section $B$ given, in the fiber of $\bb_\sigma$ over a triple $(X,\alpha,D)$, by the canonical biextension mixed Hodge structure associated to the cycles $\varSigma_\alpha$ and $\varSigma_{\alpha D}$ in $\Pic^0 X$.   Following \cite[Definition~3.3.3]{hain} the archimedean height pairing  $\langle \varSigma_\alpha, \varSigma_{\alpha D} \rangle_\infty$ for a triple $(X,\alpha,D) \in \widetilde{\jj}$  is given by the formula
\begin{equation} \label{ht_pairing}
 \langle \varSigma_\alpha, \varSigma_{\alpha D} \rangle_\infty = \log \|B\|_{\mathrm{biext}} (X,\alpha,D) \, .
\end{equation}

\subsection{Connection with Frobenius' theta function} \label{connection}
By construction the meromorphic section $B$ is regular and non-vanishing whenever the cycles $\varSigma_\alpha$ and $\varSigma_{\alpha D}$ have disjoint support. If $X$ is fixed we have that $\varSigma_\alpha$ and $\varSigma_{\alpha D}$ have disjoint support whenever $D \notin |\varTheta_{2\alpha}| \cup |F_X|$. We conclude that the support of $\divisor B$ on $\widetilde{\jj}$  is equal to $|\varTheta_{0}| \cup |F|$. A more precise analysis, cf.\ \cite[Section~4.3]{hain}, yields the equality 
\begin{equation} \label{divB} \divisor B = 2\,(F-2\,\varTheta_{0}) 
\end{equation} 
of divisors on $\widetilde{\jj}$.

Let $a$ be a characteristic in degree three. Let $\theta_a$ be the theta function with characteristic $a$ and let $\phi$ be Frobenius' theta function. As in Section~\ref{mod_prop} we set $f_a=\phi/\theta_a^2$. 
We recall that $f_\a$ transforms as a meromorphic Siegel-Jacobi form of weight seven and vanishing index.
For a suitable choice of $a$ and of uniformization of $\widetilde{\jj}$ we may view $f_a$ as a meromorphic section of the line bundle $ \ll^{\otimes 7}$ on $\widetilde{\jj}$, with divisor  
\begin{equation} \label{divf}
\divisor f_a = F - 2 \, \varTheta_{0} \, . 
\end{equation} 
Combining (\ref{divB}) and (\ref{divf}) we conclude that there exists an isomorphism of holomorphic line bundles $\chi \colon \ll^{\otimes 14} \isom \bb_\sigma$ over $\widetilde{\jj}$ that identifies the meromorphic section $f_a^{\otimes 2}$ with the meromorphic section $B$. We see that via the isomorphism $\chi$ we have a natural interpretation of the Siegel-Jacobi form $f_a^{\otimes 2}$ as a biextension variation of mixed Hodge structures associated to the Ceresa cycle in genus three. 

An application of \cite[Proposition~7.3]{dj_normal} shows that the hermitian line bundle $\bb_\sigma^{\otimes 2}$ is canonically isometric with the Hain-Reed line bundle $\nn$. We thus obtain a canonical isomorphism of line bundles $\chi^{\otimes 2} \colon \ll^{\otimes 28} \isom \nn$ on $\widetilde{\jj}$. This isomorphism can be descended to $\mm$, and yields there an explicit Hain-Reed isomorphism. 

\subsection{Applying Theorem~\ref{explicit_beta}} \label{applying}
Denote by $\|\chi\| \colon \widetilde{\jj} \to \rr$ the norm of the isomorphism $\chi$. Let $\|f_a\| $ be the norm of~$f_a$ in the Hodge metric. Then clearly we have
\begin{equation} \label{log_chi} \log \|\chi\| =  \log \|B\|_{\mathrm{biext}} - 2 \log \|f_a\| 
\end{equation}
as generalized functions on $\widetilde{\jj}$. The norm $\|\chi\|$ likewise descends to $\mm$, and we see that $2\log \|\chi\| $, when viewed as a function on $\mm$, is a representative of the Hain-Reed invariant $\beta$.  

Applying Theorem~\ref{explicit_beta} we conclude that there exists a constant $c$ such that
\begin{equation} \label{log_chi_bis}
 \log \|\chi\| = -2 \int_{\widetilde{p}} \log \|\phi\| \, \mu +4 \int_{\widetilde{p}} \log \|\theta_a\| \, \mu + c 
\end{equation}
as functions on $\widetilde{\mm}$. As $\|f_a\| = \|\phi\|/\|\theta_a\|^2$ we may rewrite (\ref{log_chi_bis}) as 
\begin{equation} \label{log_chi_thrice}
\log \|\chi\| = -2 \int_{\widetilde{p}} \log \|f_a\| \, \mu + c \, . 
\end{equation}
Combining  (\ref{log_chi}) and (\ref{log_chi_thrice}) we obtain
\begin{equation} \log \|B\|_{\mathrm{biext}}   = 2 \log \|f_a\| - 2 \int_{\widetilde{p}} \log \|f_a\| \, \mu + c \, , 
\end{equation}
as generalized functions on $\widetilde{\jj}$, which can be rewritten as
\begin{equation}  \label{logB} 
 \log\|B\|_{\mathrm{biext}} = 2 \log |f_a| - 2 \int_{\widetilde{p}} \log |f_a | \, \mu  + c \, . 
\end{equation}
It follows from our work in Section~\ref{hyp_case} that when $X$ is hyperelliptic, and $\alpha$ is a Weierstrass point on $X$, so that $2\alpha$ is the class of the hyperelliptic pencil on $X$, the divisors $F_X$ and $2\,\varTheta_{2\alpha}$ coincide and the function $f_a$ on the associated jacobian is constant. 

We find that the right hand side of (\ref{logB}) is equal to $c$ over the locus of triples $(X,\alpha,D)$ where $X$ is hyperelliptic and $\alpha$ is a Weierstrass point. Now the left hand side of (\ref{logB}) vanishes identically over this locus, as the Ceresa cycle $\varSigma_\alpha$ vanishes identically if $\alpha$ is a Weierstrass point. We deduce that $c=0$ and that we have an equality
\begin{equation} \label{logB_equality} \log \|B\|_{\mathrm{biext}} = 2 \log |f_a| - 2 \int_{\widetilde{p}} \log |f_a | \, \mu
\end{equation}
of generalized functions on $\widetilde{\jj}$. 
Combining (\ref{ht_pairing}) and (\ref{logB_equality})  we conclude that for $X$ a compact Riemann surface of genus three, for $\alpha \in \Pic^1 X$ such that $4\alpha$ is canonical and for $D \in \Pic^0 X$ arbitrary we have
\begin{equation}
 \langle \varSigma_\alpha, \varSigma_{\alpha D} \rangle_\infty   = 2 \log |f_a|(D) - 2 \int_{\Pic^0 X} \log |f_a | \, \mu  
\end{equation}
for a suitable characteristic $a$. This recovers formula (4.3) from Bost's article \cite{bostduke}. 

\bibliography{refs}

\begin{bibdiv}
\begin{biblist}

\bib{ar}{article}{
      author={Arakelov, S.~Ju.},
       title={An intersection theory for divisors on an arithmetic surface},
        date={1974},
        ISSN={0373-2436},
     journal={Izv. Akad. Nauk SSSR Ser. Mat.},
      volume={38},
       pages={1179\ndash 1192},
      review={\MR{0472815}},
}

\bib{ac}{article}{
      author={Arbarello, Enrico},
      author={Cornalba, Maurizio},
       title={The {P}icard groups of the moduli spaces of curves},
        date={1987},
        ISSN={0040-9383},
     journal={Topology},
      volume={26},
      number={2},
       pages={153\ndash 171},
         url={https://doi.org/10.1016/0040-9383(87)90056-5},
      review={\MR{895568}},
}

\bib{bost_cras}{article}{
      author={Bost, Jean-Beno\^{\i}t},
       title={Fonctions de {G}reen-{A}rakelov, fonctions th\^{e}ta et courbes
  de genre {$2$}},
        date={1987},
        ISSN={0249-6291},
     journal={C. R. Acad. Sci. Paris S\'{e}r. I Math.},
      volume={305},
      number={14},
       pages={643\ndash 646},
      review={\MR{917587}},
}

\bib{bostduke}{article}{
      author={Bost, Jean-Beno\^{\i}t},
       title={Green's currents and height pairing on complex tori},
        date={1990},
        ISSN={0012-7094},
     journal={Duke Math. J.},
      volume={61},
      number={3},
       pages={899\ndash 912},
         url={https://doi.org/10.1215/S0012-7094-90-06134-4},
      review={\MR{1084464}},
}

\bib{bmmb}{incollection}{
      author={Bost, Jean-Beno\^{\i}t},
      author={Mestre, Jean-Fran\c{c}ois},
      author={Moret-Bailly, Laurent},
       title={Sur le calcul explicite des ``classes de {C}hern'' des surfaces
  arithm\'{e}tiques de genre {$2$}},
        date={1990},
       pages={69\ndash 105},
        note={S\'{e}minaire sur les Pinceaux de Courbes Elliptiques (Paris,
  1988)},
      review={\MR{1065156}},
}

\bib{bghdj}{article}{
      author={Burgos~Gil, Jos\'{e}~Ignacio},
      author={Holmes, David},
      author={de~Jong, Robin},
       title={Singularities of the biextension metric for families of abelian
  varieties},
        date={2018},
        ISSN={2050-5094},
     journal={Forum Math. Sigma},
      volume={6},
       pages={e12, 56},
         url={https://doi.org/10.1017/fms.2018.13},
      review={\MR{3831971}},
}

\bib{cfg}{unpublished}{
      author={Cl\'ery, Fabien},
      author={Faber, Carel},
      author={van~der Geer, Gerard},
       title={Concomitants of ternary quartics and vector-valued {S}iegel and
  {T}eichm\"uller modular forms of genus three},
        date={2019},
        note={Preprint, available at
  \href{https://arxiv.org/abs/1908.04248}{ar{X}iv:1908.04248}},
}

\bib{dm}{article}{
      author={Deligne, P.},
      author={Mumford, D.},
       title={The irreducibility of the space of curves of given genus},
        date={1969},
        ISSN={0073-8301},
     journal={Inst. Hautes \'{E}tudes Sci. Publ. Math.},
      number={36},
       pages={75\ndash 109},
         url={http://www.numdam.org/item?id=PMIHES_1969__36__75_0},
      review={\MR{0262240}},
}

\bib{dhg}{article}{
      author={D'Hoker, Eric},
      author={Green, Michael~B.},
       title={Zhang-{K}awazumi invariants and superstring amplitudes},
        date={2014},
        ISSN={0022-314X},
     journal={J. Number Theory},
      volume={144},
       pages={111\ndash 150},
         url={https://doi.org/10.1016/j.jnt.2014.03.021},
      review={\MR{3239155}},
}

\bib{dhgp}{article}{
      author={D'Hoker, Eric},
      author={Green, Michael~B.},
      author={Pioline, Boris},
       title={Higher genus modular graph functions, string invariants, and
  their exact asymptotics},
        date={2019},
        ISSN={0010-3616},
     journal={Comm. Math. Phys.},
      volume={366},
      number={3},
       pages={927\ndash 979},
         url={https://doi.org/10.1007/s00220-018-3244-3},
      review={\MR{3927083}},
}

\bib{dhgrpr}{article}{
      author={d'Hoker, Eric},
      author={Green, Michael~B.},
      author={Pioline, Boris},
      author={Russo, Rodolfo},
       title={Matching the ${D}^6\mathcal{R}^4$ interaction at two-loops},
        date={2015},
        ISSN={1029-8479},
     journal={J. High Energy Phys.},
      number={1},
       pages={31},
         url={https://doi.org/10.1007/jhep01(2015)031},
}

\bib{fa}{article}{
      author={Faltings, Gerd},
       title={Calculus on arithmetic surfaces},
        date={1984},
        ISSN={0003-486X},
     journal={Ann. of Math. (2)},
      volume={119},
      number={2},
       pages={387\ndash 424},
         url={https://doi-org.proxy.library.cornell.edu/10.2307/2007043},
      review={\MR{740897}},
}

\bib{frob}{article}{
      author={Frobenius, G.},
       title={Ueber die {J}acobischen {F}unctionen dreier {V}ariabeln},
        date={1889},
        ISSN={0075-4102},
     journal={J. Reine Angew. Math.},
      volume={105},
       pages={35\ndash 100},
         url={https://doi.org/10.1515/crll.1889.105.35},
      review={\MR{1580193}},
}

\bib{vdgvg}{article}{
      author={Geemen, Bert~van},
      author={van~der Geer, Gerard},
       title={Kummer varieties and the moduli spaces of abelian varieties},
        date={1986},
        ISSN={0002-9327},
     journal={Amer. J. Math.},
      volume={108},
      number={3},
       pages={615\ndash 641},
         url={https://doi.org/10.2307/2374657},
      review={\MR{844633}},
}

\bib{gsm}{article}{
      author={Grushevsky, Samuel},
      author={Salvati~Manni, Riccardo},
       title={The {S}corza correspondence in genus 3},
        date={2013},
        ISSN={0025-2611},
     journal={Manuscripta Math.},
      volume={141},
      number={1-2},
       pages={111\ndash 124},
         url={https://doi.org/10.1007/s00229-012-0564-z},
      review={\MR{3042683}},
}

\bib{hain}{article}{
      author={Hain, Richard},
       title={Biextensions and heights associated to curves of odd genus},
        date={1990},
        ISSN={0012-7094},
     journal={Duke Math. J.},
      volume={61},
      number={3},
       pages={859\ndash 898},
         url={https://doi.org/10.1215/S0012-7094-90-06133-2},
      review={\MR{1084463}},
}

\bib{hain_normal}{incollection}{
      author={Hain, Richard},
       title={Normal functions and the geometry of moduli spaces of curves},
        date={2013},
   booktitle={Handbook of moduli. {V}ol. {I}},
      series={Adv. Lect. Math. (ALM)},
      volume={24},
   publisher={Int. Press, Somerville, MA},
       pages={527\ndash 578},
      review={\MR{3184171}},
}

\bib{hr}{article}{
      author={Hain, Richard},
      author={Reed, David},
       title={Geometric proofs of some results of {M}orita},
        date={2001},
        ISSN={1056-3911},
     journal={J. Algebraic Geom.},
      volume={10},
      number={2},
       pages={199\ndash 217},
      review={\MR{1811554}},
}

\bib{hrar}{article}{
      author={Hain, Richard},
      author={Reed, David},
       title={On the {A}rakelov geometry of moduli spaces of curves},
        date={2004},
        ISSN={0022-040X},
     journal={J. Differential Geom.},
      volume={67},
      number={2},
       pages={195\ndash 228},
         url={http://projecteuclid.org/euclid.jdg/1102536200},
      review={\MR{2153077}},
}

\bib{ig}{article}{
      author={Igusa, Jun-ichi},
       title={Modular forms and projective invariants},
        date={1967},
        ISSN={0002-9327},
     journal={Amer. J. Math.},
      volume={89},
       pages={817\ndash 855},
         url={https://doi.org/10.2307/2373243},
      review={\MR{229643}},
}

\bib{dj_adm}{article}{
      author={Jong, Robin~de},
       title={Admissible constants for genus 2 curves},
        date={2010},
        ISSN={0024-6093},
     journal={Bull. Lond. Math. Soc.},
      volume={42},
      number={3},
       pages={405\ndash 411},
         url={https://doi.org/10.1112/blms/bdp132},
      review={\MR{2651934}},
}

\bib{djsecond}{article}{
      author={Jong, Robin~de},
       title={Second variation of {Z}hang's {$\lambda$}-invariant on the moduli
  space of curves},
        date={2013},
        ISSN={0002-9327},
     journal={Amer. J. Math.},
      volume={135},
      number={1},
       pages={275\ndash 290},
         url={https://doi.org/10.1353/ajm.2013.0008},
      review={\MR{3022965}},
}

\bib{dj_normal}{article}{
      author={Jong, Robin~de},
       title={Normal functions and the height of {G}ross-{S}choen cycles},
        date={2014},
        ISSN={0027-7630},
     journal={Nagoya Math. J.},
      volume={214},
       pages={53\ndash 77},
         url={https://doi.org/10.1215/00277630-2413391},
      review={\MR{3211818}},
}

\bib{djtorus}{incollection}{
      author={Jong, Robin~de},
       title={Torus bundles and 2-forms on the universal family of {R}iemann
  surfaces},
        date={2016},
   booktitle={Handbook of {T}eichm\"{u}ller theory. {V}ol. {VI}},
      series={IRMA Lect. Math. Theor. Phys.},
      volume={27},
   publisher={Eur. Math. Soc., Z\"{u}rich},
       pages={195\ndash 227},
      review={\MR{3618190}},
}

\bib{kaw}{unpublished}{
      author={Kawazumi, Nariya},
       title={Johnson's homomorphisms and the {A}rakelov-{G}reen function},
        date={2008},
        note={Preprint, available at
  \href{https://arxiv.org/abs/0801.4218}{ar{X}iv:0801.4218}},
}

\bib{kawhandbook}{incollection}{
      author={Kawazumi, Nariya},
       title={Canonical 2-forms on the moduli space of {R}iemann surfaces},
        date={2009},
   booktitle={Handbook of {T}eichm\"{u}ller theory. {V}ol. {II}},
      series={IRMA Lect. Math. Theor. Phys.},
      volume={13},
   publisher={Eur. Math. Soc., Z\"{u}rich},
       pages={217\ndash 237},
         url={https://doi.org/10.4171/055-1/7},
      review={\MR{2497786}},
}

\bib{lock}{article}{
      author={Lockhart, P.},
       title={On the discriminant of a hyperelliptic curve},
        date={1994},
        ISSN={0002-9947},
     journal={Trans. Amer. Math. Soc.},
      volume={342},
      number={2},
       pages={729\ndash 752},
         url={https://doi.org/10.2307/2154650},
      review={\MR{1195511}},
}

\bib{mo}{incollection}{
      author={Morita, Shigeyuki},
       title={A linear representation of the mapping class group of orientable
  surfaces and characteristic classes of surface bundles},
        date={1996},
   booktitle={Topology and {T}eichm\"{u}ller spaces ({K}atinkulta, 1995)},
   publisher={World Sci. Publ., River Edge, NJ},
       pages={159\ndash 186},
      review={\MR{1659679}},
}

\bib{tata1}{book}{
      author={Mumford, David},
       title={Tata lectures on theta. {I}},
      series={Modern Birkh\"{a}user Classics},
   publisher={Birkh\"{a}user Boston, Inc., Boston, MA},
        date={2007},
        ISBN={978-0-8176-4572-4; 0-8176-4572-1},
         url={https://doi.org/10.1007/978-0-8176-4578-6},
        note={With the collaboration of C. Musili, M. Nori, E. Previato and M.
  Stillman, Reprint of the 1983 edition},
      review={\MR{2352717}},
}

\bib{pioline}{article}{
      author={Pioline, Boris},
       title={A theta lift representation for the {K}awazumi-{Z}hang and
  {F}altings invariants of genus-two {R}iemann surfaces},
        date={2016},
        ISSN={0022-314X},
     journal={J. Number Theory},
      volume={163},
       pages={520\ndash 541},
         url={https://doi.org/10.1016/j.jnt.2015.12.021},
      review={\MR{3459586}},
}

\bib{wilms}{article}{
      author={Wilms, Robert},
       title={New explicit formulas for {F}altings' delta-invariant},
        date={2017},
        ISSN={0020-9910},
     journal={Invent. Math.},
      volume={209},
      number={2},
       pages={481\ndash 539},
         url={https://doi.org/10.1007/s00222-016-0713-1},
      review={\MR{3674221}},
}

\bib{zhgs}{article}{
      author={Zhang, Shou-Wu},
       title={Gross-{S}choen cycles and dualising sheaves},
        date={2010},
        ISSN={0020-9910},
     journal={Invent. Math.},
      volume={179},
      number={1},
       pages={1\ndash 73},
         url={https://doi.org/10.1007/s00222-009-0209-3},
      review={\MR{2563759}},
}

\end{biblist}
\end{bibdiv}
\bibliographystyle{plain}

\end{document}